\documentclass{article}

\usepackage[letterpaper]{geometry}

\usepackage{natbib}

\usepackage[utf8]{inputenc} 
\usepackage[T1]{fontenc}    
\usepackage{hyperref}       
\usepackage{url}            
\usepackage{booktabs}       
\usepackage{amsfonts}       
\usepackage{nicefrac}       
\usepackage{microtype}      
\usepackage{xcolor}         

\usepackage{float}
\usepackage{algorithmic}
\usepackage{multirow}

\usepackage{subcaption}
\usepackage{graphicx}
\usepackage{latexsym}
\usepackage{amssymb}
\usepackage{amsmath}
\usepackage{amsthm}
\usepackage{thmtools}
\usepackage{forloop}
\usepackage[ruled,vlined,linesnumbered]{algorithm2e}
\usepackage{mleftright}
\mleftright

\usepackage{wrapfig}

\usepackage{caption}
\usepackage{nth}

\newtheorem{theorem}{Theorem}[section]
\newtheorem{lemma}[theorem]{Lemma}

\newtheorem{corollary}[theorem]{Corollary}
\newtheorem{definition}{Definition}[section]

\newtheorem{observation}[theorem]{Observation}

\newcommand{\defcal}[1]{\expandafter\newcommand\csname c#1\endcsname{{\mathcal{#1}}}}
\newcommand{\defbb}[1]{\expandafter\newcommand\csname b#1\endcsname{{\mathbb{#1}}}}
\newcommand{\defvec}[1]{\expandafter\newcommand\csname v#1\endcsname{{\mathbf{#1}}}}
\newcounter{calBbCounter}
\forLoop{1}{26}{calBbCounter}{
    \edef\CaptialLetter{\Alph{calBbCounter}}
		\edef\LowerLetter{\alph{calBbCounter}}
    \expandafter\defcal\CaptialLetter
		\expandafter\defbb\CaptialLetter
		\expandafter\defvec\LowerLetter
}

\newcommand{\eps}{\varepsilon}
\newcommand{\ie}{{\it i.e.}}

\newcommand{\vzero}{{\bar{0}}}
\newcommand{\vone}{{\bar{1}}}

\newcommand{\inner}[1]{ \left\langle {#1} \right\rangle }

\DeclareMathOperator{\diag}{diag}

\title{\textbf{Submodular + Concave}\\~}

\author{Siddharth Mitra\\
\emph{\normalsize{Yale University}}\\
siddharth.mitra@yale.edu
\and
Moran Feldman\\
\emph{\normalsize{University of Haifa}}\\
moranfe@cs.haifa.ac.il
\and
Amin Karbasi\\
\emph{\normalsize{Yale University}}\\
amin.karbasi@yale.edu
}

\date{}

\begin{document}

\maketitle

\begin{abstract}
  It has been well established that first order optimization methods can converge to the maximal objective value of concave functions and provide constant factor approximation guarantees for (non-convex/non-concave) continuous submodular functions. In this work, we initiate the study of the maximization of functions of the form $F(x) = G(x) +C(x)$ over a solvable convex body $P$, where $G$ is a smooth DR-submodular function and $C$ is a smooth concave function. This class of functions is a strict extension of both concave and continuous DR-submodular functions for which no theoretical guarantee is known. We provide a suite of Frank-Wolfe style algorithms, which, depending on the nature of the objective function (i.e., if $G$ and $C$ are monotone or not, and non-negative or not) and on the nature of the set $P$ (i.e., whether it is downward closed or not), provide $1-1/e$, $1/e$, or $1/2$ approximation guarantees. We then use our algorithms to get a framework to smoothly interpolate between choosing a diverse set of elements from a given ground set (corresponding to the mode of a determinantal point process) and choosing a clustered set of elements (corresponding to the maxima of a suitable concave function). Additionally, we apply our algorithms to various functions in the above class (DR-submodular + concave) in both constrained and unconstrained settings, and show that our algorithms consistently outperform natural baselines.
\end{abstract}

\section{Introduction}

Despite their simplicity, first-order optimization methods (such as gradient 
descent and its variants, Frank-Wolfe, momentum based methods, and others) have shown great success 
in many 
machine learning applications. A large body of research in the operations research 
and machine learning communities has fully demystified the 
convergence rate of such methods in minimizing well behaved convex objectives 
\citep{bubeck2015convex, nesterovbook}. 
More recently, a new surge of rigorous results has also shown that gradient 
descent methods can  find the global minima of specific non-convex 
objective functions arisen from non-negative matrix factorization \citep{arora2011computing}, robust PCA \citep{netrapalli2014nonconvex}, phase retrieval \citep{Chen_2019}, matrix completion \citep{Sun_2016}, and the training of wide neural 
networks \citep{du2019gradient, jacot2020neural, allenzhu2020learning},  to name a few. It is also very well known that finding the global minimum of a general non-convex function is indeed computationally intractable  \citep{Murty1987SomeNP}. To avoid such impossibility results, simpler goals have been pursued by the community: either developing algorithms that can escape saddle points and reach local minima \citep{ge2015escaping} or describing structural properties which guarantee that reaching a local minimizer ensures optimality \citep{sun2016nonconvex, guaranteed2017bian, hazan2015graduated}. In the same spirit, this paper quantifies  a large class of non-convex functions for which first-order optimization methods provably achieve near optimal solutions.  


More specifically, we consider objective functions that are formed by the sum of a 
continuous DR-submodular function $G(x)$ and a concave 
function $C(x)$. Recent research 
in non-convex optimization has shown that first-order optimization methods 
provide constant factor approximation guarantees for maximizing  continuous submodular 
functions \cite{guaranteed2017bian, hassani2017gradientmethods, bian2017continuous, niazadeh2018optimal, hassani2020nonconvexand, mokhtari2017conditional}. Similarly, such methods find the global maximizer of concave 
functions. However,  the class of $F(x) = G(x) +C(x)$ functions is strictly larger 
than both concave and continuous DR-submodular functions. More specifically, 
$F(x)$ is not concave nor continuous DR-submodular in general (Figure~\ref{fig:example} illustrates an example). In this paper, we 
show that first-order methods provably provide strong theoretical guarantees for 
maximizing such functions.



The combinations of continuous submodular and concave functions naturally arise 
in many ML applications such as maximizing a regularized submodular function \citep{kazemi2020regularized} or finding the mode of distributions \citep{kazemi2020regularized, robinson2019flexible}. For instance, it is common to add a regularizer to a D-optimal design objective function to  
increase the 
stability of the  final solution against perturbations \citep{he2010laplacianregularized, derezinski20experimental, banditalgsbook}. Similarly, many instances of  log-submodular distributions, such as determinantal point processes (DPPs), have been studied in depth in order to sample a diverse set of items from a ground set \citep{Kulesza_2012, Rebeschini2015fastmixing, anari2016monte}. In order to control the level of diversity, one natural recipe is to consider the combination of a log-concave (e.g., normal distribution, exponential distribution and Laplace distribution) \citep{dwivedi2019logconcave, robinson2019flexible} and log-submodular  distributions \citep{josip2014map, bresler2018learning}, i.e., $\Pr(\vx)\propto \exp(\lambda C(\vx) + (1-\lambda) G(\vx))$ for some $\lambda\in[0,1]$. In this way,  we can obtain a class of distributions that contains log-concave and log-submodular distributions as special cases. However, this class of distributions is strictly more expressive than both log-concave and log-submodular models, e.g., in contrast to log-concave distributions, they are not uni-modal in general.  Finding the mode of such distributions amounts to maximizing a combination of a continuous DR-submodular function and a concave function. 
The contributions of this paper are as follows.

\begin{itemize}
	\item Assuming first-order oracle access for the function $F$, we develop the algorithms: \textsc{Greedy Frank-Wolfe} (Algorithm \ref{alg:continuous_greedy}) and \textsc{Measured Greedy Frank-Wolfe} (Algorithm \ref{alg:measured_continuous_greedy}) which achieve constant factors approximation guarantees between $1-1/e$ and $1/e$ depending on the setting, i.e. depending on the monotonicity and non-negativity of $G$ and $C$, and depending on the constraint set having the down-closeness property or not.
	\item Furthermore, if we have access to the individual gradients of $G$ and $C$, then we are able to make the guarantee with respect to $C$ \emph{exact} using the algorithms: \textsc{Gradient Combining Frank-wolfe} (Algorithm \ref{alg:frank_wolfe}) and \textsc{Non-oblivious Frank-Wolfe} (Algorithm \ref{alg:non_ob_fw}). These results are summarized and made more precise in Table~\ref{table:summary} and Section~\ref{sec:results}.
	\item We then present experiments designed to use our algorithms to smoothly interpolate between contrasting objectives such as picking a diverse set of elements and picking a clustered set of elements. This smooth interpolation provides a way to control the amount of diversity in the final solution. We also demonstrate the use of our algorithms to maximize a large class of (non-convex/non-concave) quadratic programming problems.   
\end{itemize}

\begin{figure}[t!]
  \centering
  \includegraphics[width=\columnwidth]{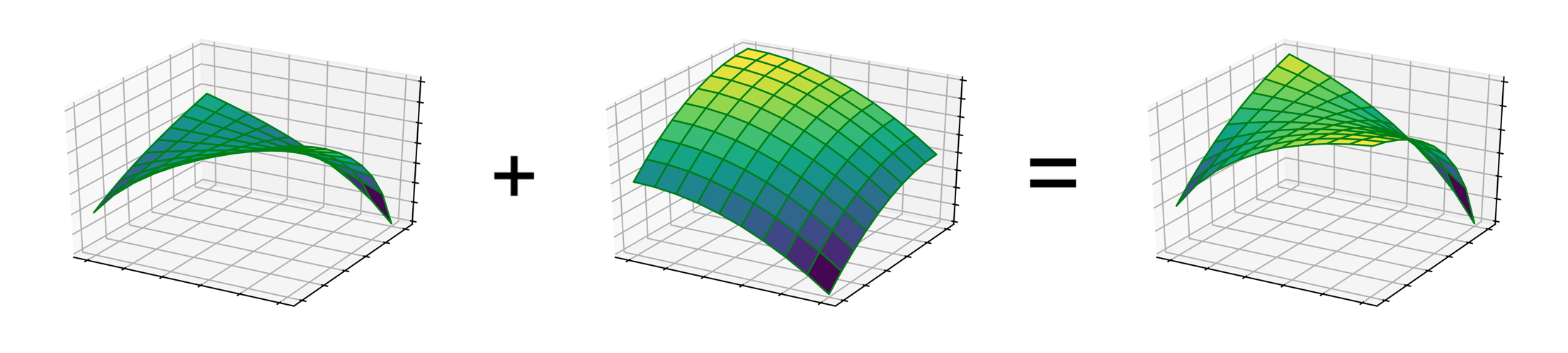}
  \caption{\textsc{Left:} (continuous) DR-submodular softmax extension. \textsc{Middle:} concave quadratic function. \textsc{Right:} sum of both.}
  \vspace{-0.3cm}
  \label{fig:example}
\end{figure}
\paragraph*{Related Work.} The study of discrete submodular maximization   has flourished in the last decade through far reaching applications in machine learning and and artificial intelligence  including viral marketing \citep{kempe2013spread}, dictionary learning \citep{krause2010dictionary}, sparse regression \citep{elenberg2017restricted}, neural network interoperability \citep{elenberg2017streaming}, crowd teaching \citep{singla2014nearoptimally}, sequential decision making \citep{alieva2021learning}, active learning \citep{wei2015active}, and data summarization \citep{mirza2013distributed}. We refer the interested reader to a recent survey by \citet{submodularsurvey} and the references therein. 
  
  Recently, \citet{guaranteed2017bian} proposed an extension of discrete submodular functions to the continuous domains that can be of use in machine learning applications. Notably, this class of (non-convex/non-concave) functions, so called continuous DR-submodular, contains  the continuous multilinear extension of discrete submodular functions \citet{maximizing2011calinescu} as a special case. Continuous DR-submodular functions can reliably model revenue maximization \citep{guaranteed2017bian}, robust budget allocation \citep{staib2017robust}, experimental design \citep{chen2018online}, MAP inference for DPPs \citep{gillenwater2012near, hassani2020stochastic}, energy allocation \citep{wilder2018}, classes of zero-sum games \citep{wilder2018game}, online welfare maximization and online task assignment \citep{sadeghi2020online}, as well as many other settings of interest.

The research on maximizing continuous DR-submodular functions in the last few years has established strong theoretical results in different optimization settings including unconstrained \citep{niazadeh2018optimal, bian2018optimal}, stochastic \cite{mokhtari2017conditional, hassani2017gradientmethods}, online \citep{chen2018online, zhang2019online, sadeghi2019online, raut2021online}, and parallel models of computation  \citep{chen2018unconstrained, mokhtari2018decentralized, xie2019decentralized, ene2019parallel}.

A different line of works study the maximization of discrete functions that can be represented as the sum of a non-negative monotone submodular function and a linear function. The ability to do so is useful in practice since the linear function can be viewed as a soft constraint, and it also has theoretical applications as is argued by the first work in this line \citep{optimal2017sviridenko} (for example, the problem of maximization of a monotone submodular function with a bounded curvature can be reduced to the maximization of the sum of a monotone submodular function and a linear function). In terms of the approximation guarantee, the algorithms suggested by \citet{optimal2017sviridenko} were optimal. However, more recent works improve over the time complexities of these algorithms \citep{guess2021feldman,submodular2019harshaw,team2020ene}, generalize them to weakly-submodular functions \citep{submodular2019harshaw}, and adapt them to other computational models such as the data stream and MapReduce models \citep{regularized2020kazemi,team2020ene}.

\section{Setting and Notation} \label{sec:setting}

Let us now formally define the setting we consider. Fix a subset $\cX$ of $R^n$ of the form $\prod_{i = 1}^n \cX_i$, where $\cX_i$ is a closed range in $\bR$. Intuitively, a function $G\colon \cX \to \bR$ is called (continuous) \emph{DR-submodular} if it exhibits diminishing returns in the sense that given a vector $\vx \in \cX$, the increase in $G(\vx)$ obtained by increasing $x_i$ (for any $i \in [n]$) by $\eps > 0$ is negatively correlated with the original values of the coordinates of $\vx$. This intuition is captured by the following definition. In this definition $\ve_i$ denotes the standard basis vector in the $i^\text{th}$ direction.
\begin{definition}[DR-submodular function]
A function $G\colon \cX \to \bR$ is DR-submodular if for every two vectors $\va, \vb \in \cX$, positive value $k$ and coordinate $i \in [n]$ we have
\[
	G(\va + k\ve_i) - G(\va) \geq G(\vb + k\ve_i) - G(\vb)
\]
whenever $\va \leq \vb$ and $\va + k\ve_i, \vb + k\ve_i \in \cX$.\footnote{Throughout the paper, inequalities between vectors should be understood as holding coordinate-wise.}
\end{definition}

It is well known that when $G$ is continuously differentiable, then it is DR-submodular if and only if $\nabla G(\va) \geq \nabla G(\vb)$ for every two vectors $\va, \vb \in \cX$ that obey $\va \leq \vb$. And when $G$ is a twice differentiable function, it is DR-submodular if and only if its hessian is non-positive at every vector $\vx \in \cX$.
Furthermore, for the sake of simplicity we assume in this work that $\cX = [0, 1]^n$. This assumption is without loss of generality because the natural mapping from $\cX$ to $[0, 1]^n$ preserves all our results.

We are interested in the problem of finding the point in some convex body $P \subseteq [0, 1]^n$ that maximizes a given function $F\colon [0, 1]^n \to \bR$ that can be expressed as the sum of a DR-submodular function $G\colon [0, 1]^n \to \bR$ and a concave function $C\colon [0, 1]^n \to \bR$. To get meaningful results for this problem, we need to make some assumptions. Here we describe three basic assumptions that we make throughout the paper. The quality of the results that we obtain improves if additional assumptions are made, as is described in Section~\ref{sec:results}.

Our first basic assumption is that $G$ is non-negative. This assumption is necessary since we obtain multiplicative approximation guarantees with respect to $G$, and such guarantees do not make sense when $G$ is allowed to take negative values.\footnote{We note that almost all the literature on submodular maximization of both discrete and continuous functions assumes non-negativity for the same reason.} Our second basic assumption is that $P$ is solvable, i.e., that one can efficiently optimize linear functions subject to it. Intuitively, this assumption makes sense because one should not expect to be able to optimize a complex function such as $F$ subject to $P$ if one cannot optimize linear functions subject to it (nevertheless, it is possible to adapt our algorithms to obtain some guarantee even when linear functions can only be approximately optimized subject to $P$). Our final basic assumption is that both functions $G$ and $C$ are $L$-smooth, which means that they are differentiable, and moreover, their gradients are $L$-Lipschitz, i.e.,
\[
	\| \nabla f(\va) - \nabla f(\vb) \|_2 \leq L \| \va - \vb \|_2
	\quad
	\forall\; \va, \vb \in [0, 1]^n
	\enspace.
\]

We conclude this section by introducing some additional notation that we need. We denote by $\vo$ an arbitrary optimal solution for the problem described above, and define $D = \max_{\vx \in P} \|\vx\|_2$. Additionally, given two vectors $\va, \vb \in \bR^n$, we denote by $\va \odot \vb$ their coordinate-wise multiplication, and by $\inner{\va,\vb}$ their standard Euclidean inner product.

%
%

\section{Main Algorithms and Results} \label{sec:results}

In this section we present our (first-order) algorithms for solving the problem described in Section~\ref{sec:setting}. In general, these algorithms are all Frank-Wolfe type algorithms, but they differ in the exact linear function which is maximized in each iteration (step 1 of the while/for loop), and in the formula used to update the solution (step 2 of the while/for loop). As mentioned previously, we assume everywhere that $G$ is a non-negative $L$-smooth DR-submodular function, $C$ is an $L$-smooth concave function, and $P$ is a solvable convex body. Some of our algorithms require additional non-negativity and/or monotonicity assumptions on the functions $G$ and $C$, and occasionally they also require a downward closed assumption on $P$. A summary of which settings each algorithm is applicable to can be found in Table~\ref{table:summary}. Each algorithm listed in the table outputs a point $x \in P$ which is guaranteed to obey $F(x) \geq \alpha \cdot G(\vo) + \beta \cdot C(\vo) - E$ for the constants $\alpha$ and $\beta$ given in Table~\ref{table:summary} and some error term $E$.

\begin{table*}[t]
\caption{Summary of algorithms, settings, and guarantees (``\textsc{non-neg.}'' is a shorthand for ``non-negative''). All of the conditions stated in the table are in addition to the continuity and smoothness of $G$ and $C$, and the convexity of $P$. $\alpha$ and $\beta$ are the constants preceding $G(\vo) $ and $ C(\vo)$ respectively in the lower bound on the output of the algorithm.}
\label{table:summary}
\vskip 0.15in
\begin{center}
\begin{small}
\begin{sc}
\begin{tabular}{lllccc}
\toprule
Algorithm (Section) & \multicolumn{1}{c}{$G$} & \multicolumn{1}{c}{$C$} & $P$ & $\alpha$ & $\beta$\\
\midrule
Greedy & monotone & monotone & \multirow{2}{*}{general} & \multirow{2}{*}{$1-1/e$}& \multirow{2}{*}{$1-1/e$}\\
Frank-Wolfe & \multicolumn{1}{r}{\hspace{1mm}\& non-neg.} & \multicolumn{1}{r}{\hspace{1mm}\& non-neg.} &&&\\[1mm]
Measured Greedy & monotone & \multirow{2}{*}{non-neg.} & \multirow{2}{*}{down-closed} & \multirow{2}{*}{$1-1/e$} & \multirow{2}{*}{$1/e$} \\
Frank-Wolfe & \multicolumn{1}{r}{\& non-neg.} &&&& \\[1mm]
Measured Greedy & \multirow{2}{*}{non-neg.} & monotone & \multirow{2}{*}{down-closed} & \multirow{2}{*}{$1/e$} & \multirow{2}{*}{$1-1/e$}\\
Frank-Wolfe && \multicolumn{1}{r}{\& non-neg.} &&& \\[1mm]
Measured Greedy & monotone & monotone & \multirow{2}{*}{down-closed} & \multirow{2}{*}{$1-1/e$} & \multirow{2}{*}{$1-1/e$}\\
Frank-Wolfe & \multicolumn{1}{r}{\& non-neg.} & \multicolumn{1}{r}{\& non-neg.} &&&\\[1mm]
Measured Greedy & \multirow{2}{*}{non-neg.} & \multirow{2}{*}{non-neg.} & \multirow{2}{*}{down-closed} & \multirow{2}{*}{$1/e$} & \multirow{2}{*}{$1/e$}\\
Frank-Wolfe &&&&&\\[1mm]
Gradient Combining & monotone & \multirow{2}{*}{General} & \multirow{2}{*}{General} & \multirow{2}{*}{$1/2 - \eps$} & \multirow{2}{*}{1}\\
Frank-Wolfe & \multicolumn{1}{r}{\& non-neg.} &&&& \\[1mm]
Non-Oblivious & monotone & \multirow{2}{*}{non-neg.} & General & \multirow{2}{*}{\hspace{-1mm}$1-1/e - \eps$} & \multirow{2}{*}{$1-\eps$}\\
Frank-Wolfe & \multicolumn{1}{r}{\& non-neg.}\\
\bottomrule
\end{tabular}
\end{sc}
\end{small}
\end{center}
\vskip -0.1in
\end{table*}

\subsection{Greedy Frank-Wolfe Algorithm} \label{sec:greedyfw}

In this section we assume that both $G$ and $C$ are monotone and non-negative functions (in addition to their other properties). Given this assumption, we analyze the guarantee of the greedy Frank-Wolfe variant appearing as Algorithm~\ref{alg:continuous_greedy}. This algorithm is related to the Continuous Greedy algorithm for discrete objective functions due to~\citet{maximizing2011calinescu}, and it gets a quality control parameter $\eps \in (0, 1)$. We assume in the algorithm that $\eps^{-1}$ is an integer. This assumption is without loss of generality because, if $\eps$ violates the assumption, then it can be replaced with a value from the range $[\eps/2, \eps]$ that obeys it. 

\begin{algorithm}[H]
\caption{\textsc{Greedy Frank-Wolfe} $(\eps)$} \label{alg:continuous_greedy}
\begin{algorithmic}
\STATE Let $t \gets 0$ and $\vy^{(t)} \gets \vzero$
\WHILE{$t<1$}
\STATE $\vs^{(t)} \gets \arg \max_{\vx \in P} \inner{\nabla F(\vy^{(t)}) , \vx}$
\STATE $\vy^{(t + \eps)} \gets \vy^{(t)} + \eps \cdot \vs^{(t)}$
\STATE $t \gets t + \eps$
\ENDWHILE
\RETURN{$\vy^{(1)}$}
\end{algorithmic}
\end{algorithm}

One can observe that the output $\vy^{(1)}$ of Algorithm~\ref{alg:continuous_greedy} is within the convex body $P$ because it is a convex combination of the vectors $\vs^{(0)}, \vs^{(\eps)}, \vs^{(2\eps)}, \dotsc, \vs^{1 - \eps}$, which are all vectors in $P$. Let us now analyze the value of the output of Algorithm~\ref{alg:continuous_greedy}. The next lemma is the first step towards this goal. It provides a lower bounds on the increase in the value of $\vy^{(t)}$ as a function of $t$.
\begin{lemma} \label{lem:step_continuous_greedy}
For every $0 \leq i < \eps^{-1}$, $F(\vy^{(\eps(i + 1))}) - F(\vy^{(\eps i)}) \geq \eps \cdot [F(\vo) - F(\vy^{(\eps i)})] - \eps^2 L D^2$.
\end{lemma}

\begin{proof}
Observe that
\begin{align*}
	F(\vy^{(\eps(i + 1))}) - F(&\vy^{(\eps i)})
	=
	F(\vy^{(\eps i)} + \eps \cdot \vs^{(\eps i)}) - F(\vy^{(\eps i)})
	=
	\int_0^\eps \left. \frac{d F(\vy^{(\eps i)} + z \cdot \vs^{(\eps i)})}{d z} \right|_{z = r} dr
	\\
	={} &
	\int_0^\eps \inner{\vs^{(\eps i)}, \nabla F(\vy^{(\eps i)} + r \cdot \vs^{(\eps i)})} dr
	\geq
	\int_0^\eps \left\{\inner{\vs^{(\eps i)}, \nabla F(\vy^{(\eps i)})} - 2r L D^2 \right\} dr\\
	={} &
	\inner{\eps \vs^{(\eps i)}, \nabla F(\vy^{(\eps i)})} - \eps^2 L D^2
	\geq
	\inner{\eps \vo, \nabla F(\vy^{(\eps i)})} - \eps^2 L D^2
	\enspace.
\end{align*}
where the last inequality follows since $\vs^{(\eps i)}$ is the maximizer found by Algorithm~\ref{alg:continuous_greedy}.
Therefore, to prove the lemma it only remains to show that $\inner{\vo, \nabla F(\vy^{(\eps i)})} \geq F(\vo) - F(\vy^{(\eps i)})$.

Since $C$ is concave and monotone,
\[
	C(\vo)
	\leq
	C(\vo \vee \vy^{(\eps i)})
	\leq
	C(\vy^{(\eps i)}) + \inner{\vo \vee \vy^{(\eps i)} - \vy^{(\eps i)}, \nabla C(\vy^{(\eps i)})}
	\leq
	C(\vy^{(\eps i)}) + \inner{\vo, \nabla C(\vy^{(\eps i)})}
	\enspace.
\]
Similarly, since $G$ is DR-submodular and monotone,
\begin{align*}
	G(\vo)
	\leq{} &
	G(\vo \vee \vy^{(\eps i)})
	=
	G(\vy^{(\eps i)}) + \int_0^1 \left. \frac{d G(y^{(\eps i)} + z \cdot (\vo \vee \vy^{(\eps i)} - \vy^{(\eps i)}))}{dz} \right|_{z = r} dr\\
	={} &
	G(\vy^{(\eps i)}) + \int_0^1 \inner{\vo \vee \vy^{(\eps i)} - \vy^{(\eps i)}, \nabla G(\vy^{(\eps i)} + r \cdot (\vo \vee \vy^{(\eps i)} - \vy^{(\eps i)}))} dr\\
	\leq{} &
	G(\vy^{(\eps i)}) + \inner{\vo \vee \vy^{(\eps i)} - \vy^{(\eps i)}, \nabla G(\vy^{(\eps i)})}
	\leq
	G(\vy^{(\eps i)}) + \inner{\vo, \nabla G(\vy^{(\eps i)})}
	\enspace.
\end{align*}
Adding the last two inequalities and rearranging gives the inequality $\inner{\vo, \nabla F(\vy^{(\eps i)})} \geq F(\vo) - F(\vy^{(\eps i)})$ that we wanted to prove.
\end{proof}

The corollary below follows by showing (via induction) that the inequality $F(\vy^{(\eps i)})
	\geq
	\big[1 - \left(1 - \eps\right)^i\big] \cdot F(\vo) - i \cdot \eps^2 LD^2 $ holds for every integer $0 \leq i \leq \eps^{-1}$, and then plugging in $i = \eps^{-1}$. The inductive step is proven using Lemma~\ref{lem:step_continuous_greedy}.

\begin{corollary} \label{cor:continuous_greedy_guarantee}
$F(\vy^{(1)}) \geq (1 - e^{-1}) \cdot F(\vo) - O(\eps LD^2)$.
\end{corollary}

\begin{proof}
We prove by induction that for every integer $0 \leq i \leq \eps^{-1}$ we have
\begin{equation} \label{eq:induction_claim}
	F(\vy^{(\eps i)})
	\geq
	\left[1 - \left(1 - \eps\right)^i\right] \cdot F(\vo) - i \cdot \eps^2 LD^2
	\enspace.
\end{equation}
One can note that the corollary follows from this inequality by plugging $i = \eps^{-1}$ since $(1 - \eps)^{\eps^{-1}} \leq e^{-1}$.

For $i = 0$, Inequality~\eqref{eq:induction_claim} follows from the non-negativity of $F$ (which follows from the non-negativity of $G$ and $C$). Next, let us prove Inequality~\eqref{eq:induction_claim} for $1 \leq i \leq \eps^{-1}$ assuming it holds for $i - 1$. By Lemma~\ref{lem:step_continuous_greedy},
\begin{align*}
	F(\vy^{(\eps i)})
	\geq{} &
	F(\vy^{(\eps (i - 1))}) + \eps \cdot [F(\vo) - F(\vy^{(\eps (i - 1))})] - \eps^2 L D^2\\
	={} &
	(1 - \eps) \cdot F(\vy^{(\eps (i - 1))}) + \eps \cdot F(\vo) - \eps^2 L D^2\\
	\geq{} &
	(1 - \eps) \cdot \left\{\left[1 - \left(1 - \eps\right)^{i - 1}\right] \cdot F(\vo) - (i - 1) \cdot \eps^2 LD^2\right\} + \eps \cdot F(\vo) - \eps^2 L D^2\\
	\geq{} &
	\left[1 - \left(1 - \eps\right)^i\right] \cdot F(\vo) - i \cdot \eps^2 L D^2
	\enspace,
\end{align*}
where the second inequality holds due to the induction hypothesis.
\end{proof}

We are now ready to summarize the properties of Algorithm~\ref{alg:continuous_greedy} in a theorem.
\begin{theorem} \label{thm:continuous_greedy}
Let $S$ be the time it takes to find a point in $P$ maximizing a given liner function, then Algorithm~\ref{alg:continuous_greedy} runs in $O(\eps^{-1}(n + S))$ time, makes $O(1/\eps)$ calls to the gradient oracle, and outputs a vector $\vy$ such that $F(\vy) \geq (1 - 1/e) \cdot F(\vo) - O(\eps L D^2)$.
\end{theorem}

\begin{proof}
The output guarantee of Theorem \ref{thm:continuous_greedy} follows directly from Corollary~\ref{cor:continuous_greedy_guarantee}. The time and oracle complexity follows by observing that the algorithm's \textbf{while} loop makes $\eps^{-1}$ iterations, and each iteration requires $O(n + S)$ time, in addition to making a single call to the gradient oracle.
\end{proof}

\subsection{Measured Greedy Frank-Wolfe Algorithm}\label{section:measured_greedyfw}

In this section we assume that $P$ is a down-closed body (in addition to being convex) and that $G$ and $C$ are both non-negative. Given these assumptions, we analyze the guarantee of the variant of Frank-Wolfe appearing as Algorithm~\ref{alg:measured_continuous_greedy}, which is motivated by the Measured Continuous Greedy algorithm for discrete objective functions due to~\citet{feldman2011greedy}. We again have a quality control parameter $\eps \in (0, 1)$, and assume (without loss of generality) that $\eps^{-1}$ is an integer. 

\begin{algorithm}[H]
\caption{\textsc{Measured Greedy Frank-Wolfe} $(\eps)$} \label{alg:measured_continuous_greedy}
\begin{algorithmic}
\STATE Let $t \gets 0$ and $\vy^{(t)} \gets \vzero$
\WHILE{$t<1$}
\STATE $\vs^{(t)} \gets \arg \max_{\vx \in P} \langle (\vone - \vy^{(t)}) \odot \nabla F(\vy^{(t)}) , \vx \rangle$
\STATE $\vy^{(t + \eps)} \gets \vy^{(t)} + \eps \cdot (\vone - \vy^{(t)}) \odot \vs^{(t)}$
\STATE $t \gets t + \eps$
\ENDWHILE
\RETURN{$\vy^{(1)}$}
\end{algorithmic}
\end{algorithm}

We begin the analysis of Algorithm~\ref{alg:measured_continuous_greedy} by bounding the range of the entries of the vectors $\vy^{(t)}$, which is proven by induction.
\begin{lemma} \label{lem:y_bound}
For every two integers $0 \leq i \leq \eps^{-1}$ and $1 \leq j \leq n$, $0 \leq y^{(\eps i)}_j \leq 1 - (1 - \eps)^i \leq 1$.
\end{lemma}
\begin{proof}
We prove the lemma by induction on $i$. For $i = 0$ the lemma is trivial since $y^{(0)}_j = 0$ by the initialization of $\vy^{(0)}$. Assume now that the lemma holds for $i - 1$, and let us prove it for $1 \leq i \leq \eps^{-1}$. Observe that
\[
	y^{(\eps i)}_j
	=
	y^{(\eps (i - 1))}_j + \eps \cdot (1 - y^{(\eps (i - 1))}_j) \cdot s^{(\eps i)}_j
	\geq
	y^{(\eps (i - 1))}_j
	\geq
	0
	\enspace,
\]
where the first inequality holds since $y^{(\eps (i - 1))}_j \leq 1$ by the induction hypothesis and $s^{(\eps i)}_j \geq 0$ since $\vs^{(\eps i)} \in P$. Similarly,
\begin{align*}
	y^{(\eps i)}_j
	={} &
	y^{(\eps (i - 1))}_j + \eps \cdot (1 - y^{(\eps (i - 1))}_j) \cdot s^{(\eps i)}_j
	\leq
	y^{(\eps (i - 1))}_j + \eps \cdot (1 - y^{(\eps (i - 1))}_j)\\
	={} &
	\eps + (1 - \eps) \cdot y^{(\eps (i - 1))}_j
	\leq
	\eps + (1 - \eps) \cdot [1 - (1 - \eps)^{i - 1}]
	=
	1 - (1 - \eps)^i
	\enspace,
\end{align*}
where the first inequality holds because $s^{(\eps i)}_j \leq 1$ since $\vs^{(\eps i)} \in P$, and the second inequality holds by the induction hypothesis.
\end{proof}

Using the last lemma, we can see why the output of Algorithm \ref{alg:measured_continuous_greedy} must be in $P$. Recall that the output of Algorithm~\ref{alg:measured_continuous_greedy} is
$
	\vy^{(1)}
	=
	\sum_{i = 1}^{\eps^{-1}} \eps \cdot (\vone - \vy^{\eps (i - 1)}) \odot \vs^{(\eps i)}
	\leq
	\eps \cdot \sum_{i = 1}^{\eps^{-1}} \vs^{(\eps i)}, $ where the inequality follows due to Lemma \ref{lem:y_bound} and the fact that $\vs^{(\eps i)}$ (as a vector in $P$) is non-negative. Since $\eps \cdot \sum_{i = 1}^{\eps^{-1}} \vs^{(\eps i)}$ is a convex combination of points in $P$, it also belongs to $P$. Since the vector $\vy^{(1)}$ is coordinate-wise dominated by this combination, the down-closeness of $P$ implies $\vy^{(1)} \in P$.

Our next objective is to prove an approximation guarantee for Algorithm~\ref{alg:measured_continuous_greedy}. Towards this goal, let us define, for every function $H$, the expression $M(H, i)$ to be $1$ when $H$ is monotone and $(1 - \eps)^i$ when $H$ is non-monotone. Using this definition, we can state the following lemma and corollary, which are counterparts of Lemma~\ref{lem:step_continuous_greedy} and Corollary~\ref{cor:continuous_greedy_guarantee} from Section~\ref{sec:greedyfw}.
\begin{lemma} \label{lem:step_measured_continuous_greedy}
For every integer $0 \leq i < \eps^{-1}$, $F(\vy^{(\eps(i + 1))}) - F(\vy^{(\eps i)}) \geq \eps \cdot \{[M(G, i) \cdot G(\vo) - G(\vy^{(\eps i)})] + [M(C, i) \cdot C(\vo) - C(\vy^{(\eps i)})]\} - \eps^2 L D^2$.
\end{lemma}
\begin{proof}
Observe that
\begin{align*}
	F(\vy^{(\eps(i + 1))}) - F(\vy^{(\eps i)})
	={} &
	F(\vy^{(\eps i)} + \eps \cdot (\vone - \vy^{(\eps i)}) \odot \vs^{(\eps i)}) - F(\vy^{(\eps i)})\\
	={} &
	\int_0^\eps \left. \frac{d F(\vy^{(\eps i)} + z \cdot (\vone - \vy^{(\eps i)}) \odot \vs^{(\eps i)})}{d z} \right|_{z = r} dr\\
	={} &
	\int_0^\eps \inner{(\vone - \vy^{(\eps i)}) \odot \vs^{(\eps i)}, \nabla F(\vy^{(\eps i)} + r \cdot (\vone - \vy^{(t)}) \odot \vs^{(\eps i)})} dr\\
	\geq{} &
	\int_0^\eps \left\{\inner{((\vone - \vy^{(\eps i)}) \odot \vs^{(\eps i)}, \nabla F(\vy^{(\eps i)})} - 2r L D^2 \right\} dr\\
	={} &
	\inner{\eps (\vone - \vy^{(\eps i)}) \odot \vs^{(\eps i)}, \nabla F(\vy^{(\eps i)})} - \eps^2 L D^2
	\enspace.
\end{align*}
Furthermore, we also have
\begin{align*}
	\inner{(\vone - \vy^{(\eps i)}) \odot \vs^{(\eps i)}, \nabla F(\vy^{(\eps i)})}
	={} &
	\inner{\vs^{(\eps i)}, (\vone - \vy^{(\eps i)}) \odot \nabla F(\vy^{(\eps i)})}\\
	\geq{} &
	\inner{\vo, (\vone - \vy^{(\eps i)}) \odot \nabla F(\vy^{(\eps i)})}
	\enspace,
\end{align*}
where the inequality holds since $s^{(\eps i)}$ is the maximizer found by Algorithm~\ref{alg:measured_continuous_greedy}. Combining the last two inequalities, we get
\begin{equation} \label{eq:lower_bound_level_1}
	F(\vy^{(\eps(i + 1))}) - F(\vy^{(\eps i)})
	\geq
	\inner{\eps \vo, (\vone - \vy^{(\eps i)}) \odot \nabla F(\vy^{(\eps i)})} - \eps^2 L D^2
	\enspace.
\end{equation}

Let us now find a lower bound on the expression $\inner{\vo, (\vone - \vy^{(\eps i)}) \odot \nabla F(\vy^{(\eps i)})}$ in the above inequality. Since $C$ is concave,
\begin{align*}
	C(\vo \odot (\vone - \vy^{(\eps i)}) + \vy^{(\eps i)}) - C(\vy^{(\eps i)})
	\leq{} &
	\inner{\vo \odot (\vone - \vy^{(\eps i)}), \nabla C(\vy^{(\eps i)})}\\
	={} &
	\inner{\vo, (\vone - \vy^{(\eps i)}) \odot \nabla C(\vy^{(\eps i)})}
	\enspace.
\end{align*}
Similarly, since $G$ is DR-submodular,
\begin{align*}
	G(\vo \odot (\vone - \vy^{(\eps i)}) + \vy^{(\eps i)}) &{}- G(\vy^{(\eps i)})
	=
	\int_0^1 \left. \frac{d G(\vy^{(\eps i)} + z \cdot \vo \odot (\vone - \vy^{(\eps i)})}{dz} \right|_{z = r} dr\\
	={} &
	\int_0^1 \inner{\vo \odot (\vone - \vy^{(\eps i)}), \nabla G(\vy^{(\eps i)} + r \cdot \vo \odot (\vone - \vy^{(\eps i)}))} dr\\
	\leq{} &
	\inner{\vo \odot (\vone - \vy^{(\eps i)}), \nabla G(\vy^{(\eps i)})}
	\leq
	\inner{\vo, (\vone - \vy^{(t)}) \odot \nabla G(\vy^{(\eps i)})}
	\enspace.
\end{align*}
Adding the last two inequalities provides the promised lower bound on $\inner{\vo, (\vone - \vy^{(t)}) \odot \nabla F(\vy^{(\eps i)})}$. Plugging this lower bound into Inequality~\eqref {eq:lower_bound_level_1} yields
\begin{align*}
	F(\vy^{(\eps(i + 1))}) - F(\vy^{(\eps i)})
	\geq
	\eps \cdot \{&[G(\vo \odot (\vone - \vy^{(\eps i)}) + \vy^{(\eps i)}) - G(\vy^{(\eps i)})] \\& + [C(\vo \odot (\vone - \vy^{(\eps i)}) + \vy^{(\eps i)}) - C(\vy^{(\eps i)})]\} - \eps^2 L D^2
	\enspace.
\end{align*}

Given the last inequality, to prove the lemma it remains to show that $G(\vo \odot (\vone - \vy^{(\eps i)}) + \vy^{(\eps i)}) \geq M(G, i) \cdot G(\vo)$ and $C(\vo \odot (\vone - \vy^{(\eps i)}) + \vy^{(\eps i)}) \geq M(C, i) \cdot C(\vo)$. Since $\vo \odot (\vone - \vy^{(\eps i)}) + \vy^{(\eps i)} \geq \vo$, these inequalities follow immediately when $G$ and $C$ are monotone, respectively. Therefore, we concentrate on proving these inequalities when $G$ and $C$ are non-monotone. Since $C$ is concave,
\begin{align*}
	C(\vo \odot (\vone - {}&\vy^{(\eps i)}) + \vy^{(\eps i)})\\
	={} &
	C\left(\mspace{-6mu}M(C, i) \cdot \vo + (1 - M(C, i)) \cdot \frac{((1 - M(C, i)) \cdot \vone - \vy^{(\eps i)}) \odot \vo + \vy^{(\eps i)}}{1 - M(C, i)}\right)\\
	\geq{} &
	M(C, i) \cdot C(\vo) + (1 - M(C, i)) \cdot C\left(\frac{((1 - M(C, i)) \cdot \vone - \vy^{(\eps i)}) \odot \vo + \vy^{(\eps i)}}{1 - M(C, i)}\right)\\
	\geq{} &
	M(C, i) \cdot C(\vo)
	\enspace,
\end{align*}
where the second inequality follows from the non-negativity of $C$. We also note that
\[
	\frac{((1 - M(C, i)) \cdot \vone - \vy^{(\eps i)}) \odot \vo + \vy^{(\eps i)}}{1 - M(C, i)} \in [0, 1]^n
\]
since $0 \leq \vy^{(\eps i)} \leq (1 - M(C, i)) \cdot \vone$ by Lemma~\ref{lem:y_bound}. Similarly, since $G$ is DR-submodular,
\begin{align*}
	G(\vo \odot (\vone & {} - \vy^{(\eps i)}) + \vy^{(\eps i)})
	=
	G(\vo + (\vone - \vo) \odot \vy^{(\eps i)})\\
	={} &
	G(\vo) + \int_{0}^{1} \left. \frac{dG(\vo + z \cdot (\vone - \vo) \odot \vy^{(\eps i)})}{dz} \right|_{z = r} dr\\
	\geq{} &
	G(\vo) + (1 - M(G, i)) \cdot \int_{0}^{(1 - M(G, i))^{-1}} \left. \frac{dG(\vo + z \cdot (\vone - \vo) \odot \vy^{(\eps i)})}{dz} \right|_{z = r} dr\\
	={} &
	M(G, i) \cdot G(\vo) + (1 - M(G, i)) \cdot G\left(\vo + \frac{1}{1 - M(G, i)} \cdot (\vone - \vo) \odot \vy^{(\eps i)}\right)\\
	\geq{} &
	M(G, i) \cdot G(\vo)
	\enspace,
\end{align*}
where the second inequality follows from the non-negativity of $G$, and the first inequality holds since
\[
	\vo + \frac{1}{1 - M(G, i)} \cdot (\vone - \vo) \odot \vy^{(\eps i)} \in [0, 1]^n
\]
because $0 \leq \vy^{(\eps i)} \leq (1 - M(C, i)) \cdot \vone$ by Lemma~\ref{lem:y_bound}.
\end{proof}

\begin{corollary} \label{cor:measured_continuous_greedy_guarantee}
\begin{equation*}
F(\vy^{(1)})
	\mspace{-3mu}\geq\mspace{-3mu}
	\left\{\begin{array}{ll}
		\mspace{-8mu} 1 - e^{-1} \mspace{-7mu} & \text{if $G$ is monotone} \mspace{-8mu} \\
		\mspace{-8mu} e^{-1} & \text{otherwise}
	\end{array}\right\}
	\cdot G(\vo)
	+
	\left\{\begin{array}{ll}
		\mspace{-8mu} 1 - e^{-1} \mspace{-7mu} & \text{if $C$ is monotone} \mspace{-8mu} \\
		\mspace{-8mu} e^{-1} & \text{otherwise}
	\end{array}\right\}
	\cdot C(\vo) - O(\eps LD^2)
	\mspace{-2mu}\enspace.
\end{equation*}
\end{corollary}
\begin{proof}
For every function $H$, let $S(H, i)$ be $1 - (1 - \eps)^i$ if $H$ is monotone, and $(\eps i) \cdot (1 - \eps)^{i - 1}$ if $H$ is non-monotone. We prove by induction that for every integer $0 \leq i \leq \eps^{-1}$ we have
\begin{equation} \label{eq:induction_claim_measured}
	F(\vy^{(\eps i)})
	\geq
	S(G, i) \cdot G(\vo) + S(C, i) \cdot C(\vo) - i \cdot \eps^2 LD^2
	\enspace.
\end{equation}
One can note that the corollary follows from this inequality by plugging $i = \eps^{-1}$ since $S(H, \eps^{-1}) = 1 - (1 - \eps)^{\eps^{-1}} \geq 1 - e^{-1}$ when $H$ is monotone and $S(H, \eps^{-1}) = (1 - \eps)^{\eps^{-1} - 1} \geq e^{-1}$ when $H$ is non-monotone.

For $i = 0$, Inequality~\eqref{eq:induction_claim_measured} follows from the non-negativity of $F$ since $S(H, 0) = 0$ both when $H$ is monotone and non-monotone. Next, let us prove Inequality~\eqref{eq:induction_claim_measured} for $1 \leq i \leq \eps^{-1}$ assuming it holds for $i - 1$. We note that $(1 - \eps) \cdot S(H, i - 1) + \eps \cdot M(H, i - 1) = S(H, i)$. For a monotone $H$ this is true since
\[
	(1 - \eps) \cdot S(H, i - 1) + \eps \cdot M(H, i - 1)
	=
	(1 - \eps) \cdot [1 - (1 - \eps)^{i - 1}] + \eps
	=
	1 - (1 - \eps)^{i}
	=
	S(H, i)
	\enspace,
\]
and for a non-monotone $H$ this is true since
\begin{align*}
	(1 - \eps) \cdot S(H, i - 1) + \eps \cdot M(H, i - 1)
	={} &
	(1 - \eps) \cdot (\eps (i - 1)) \cdot (1 - \eps)^{i - 2} + \eps \cdot (1 - \eps)^{i - 1}\\
	={} &
	(\eps i) \cdot (1 - \eps)^{i - 1}
	=
	S(H, i)
	\enspace.
\end{align*}
Using Lemma~\ref{lem:step_measured_continuous_greedy}, we now get
\begin{align*}
	F(\vy^{(\eps i)})
	\geq{} &
	F(\vy^{(\eps (i - 1))}) + \eps \cdot \{[M(G, i - 1) \cdot G(\vo) - G(\vy^{(\eps (i - 1))})] \\&\mspace{200mu}+ [M(C, i - 1) \cdot C(\vo) - C(\vy^{(\eps (i - 1))})]\} - \eps^2 L D^2\\
	={} &
	(1 - \eps) \cdot F(\vy^{(\eps (i - 1))}) + \eps \cdot \{M(G, i - 1) \cdot G(\vo) + M(C, i - 1) \cdot C(\vo)\}- \eps^2 L D^2\\
	\geq{} &
	[(1 - \eps) \cdot S(G, i - 1) \cdot G(\vo) + \eps \cdot M(G, i - 1) \cdot G(\vo)] \\&\mspace{200mu}+ [(1 - \eps) \cdot S(C, i - 1) \cdot C(\vo) + \eps \cdot M(C, i - 1) \cdot C(\vo)]\\
	&\mspace{200mu}- [(1 - \eps) \cdot (i - 1) \cdot \eps^2 LD^2 + \eps^2 L D^2]\\
	\geq{} &
	S(G, i) \cdot G(\vo) + S(C, i) \cdot C(\vo) - i \cdot \eps^2 L D^2
	\enspace,
\end{align*}
where the second inequality holds due to the induction hypothesis.
\end{proof}

We are now ready to state and prove our main theorem for Algorithm~\ref{alg:measured_continuous_greedy}.

\begin{restatable}{theorem}{thmMeasuredContinuousGreedy} \label{thm:measured_continuous_greedy}
Let $S$ be the time it takes to find a point in $P$ maximizing a given liner function, then Algorithm~\ref{alg:measured_continuous_greedy} runs in $O(\eps^{-1}(n + S))$ time, makes $ O(1/\eps )$ calls to the gradient oracle, and outputs a vector $\vy$ such that
\begin{equation*}
	F(\vy)
	\geq
	\left\{\begin{array}{ll}
		\mspace{-8mu} 1 - e^{-1} \mspace{-7mu} & \text{if $G$ is monotone} \mspace{-8mu} \\
		\mspace{-8mu} e^{-1} & \text{otherwise}
	\end{array}\right\}
	\cdot G(\vo)
	+
	\left\{\begin{array}{ll}
		\mspace{-8mu} 1 - e^{-1} \mspace{-7mu} & \text{if $C$ is monotone} \mspace{-8mu} \\
		\mspace{-8mu} e^{-1} & \text{otherwise}
	\end{array}\right\}
	\cdot C(\vo) - O(\eps LD^2)
	\enspace.
\end{equation*}
\end{restatable}

\begin{proof}
The output guarantee for Theorem~\ref{thm:measured_continuous_greedy} follows from Corollary~\ref{cor:measured_continuous_greedy_guarantee}. Like in the analysis of Theorem~\ref{thm:continuous_greedy}, the time and oracle complexities follow from the observation that the algorithm performs $\eps^{-1}$ iterations, and each iteration takes $O(n + S)$ time and makes a single call to the gradient oracle.
\end{proof}


\noindent \textbf{Remark:} The guarantees of Algorithms~\ref{alg:continuous_greedy} and~\ref{alg:measured_continuous_greedy} apply in general to different settings. However, both guarantees apply in the special case in which the functions $G$ and $C$ are monotone and the polytope $P$ is down-closed. Interestingly, the two guarantees are identical in this common special case.

\subsection{Gradient Combining Frank-Wolfe Algorithm} \label{section:grad_combfw}

Up to this point, the guarantees of the algorithms that we have seen had both $\alpha$ and $\beta$ that are strictly smaller than $1$. However, since concave functions can be exactly maximized, it is reasonable to expect also algorithms for which the coefficient $\beta$ associated with $C(\vo)$ is equal to $1$. In Sections~\ref{section:grad_combfw} and~\ref{section:nonoblivious_fw}, we describe such algorithms.


In this section, we assume that $G$ is a monotone and non-negative function (in addition to its other properties). The algorithm we study in this section is Algorithm~\ref{alg:frank_wolfe}, and it again takes a quality control parameter $\eps \in (0, 1)$ as input. This time, however, the algorithm assumes that $\eps^{-3}$ is an integer. As usual, if that is not the case, then $\eps$ can be replaced with a value from the range $[\eps/2, \eps]$ that has this property. 

\begin{algorithm}[H]
\caption{\textsc{Gradient Combining Frank-Wolfe} $(\eps)$} \label{alg:frank_wolfe}
\begin{algorithmic}
\STATE Let $\vy^{(0)}$ be a vector in $P$ maximizing $C$ up to an error of $\eta \geq 0$.
\FOR{$i = 1$ \textbf{to} $\eps^{-3}$}
\STATE $\vs^{(i)} \gets \arg \max_{\vx \in P} \inner{\nabla G(\vy^{(i - 1)}) + 2\nabla C(\vy^{(i - 1)}) , \vx}$
\STATE $\vy^{(i)} \gets (1 - \eps^2) \cdot \vy^{(i - 1)} + \eps^2 \cdot \vs^{(i)}$
\ENDFOR
\RETURN{the vector maximizing $F$ among $\{\vy^{(0)},\dotsc,\allowbreak \vy^{(\eps^{-3})}$\}}
\end{algorithmic}
\end{algorithm}

Firstly, note that for every integer $0 \leq i \leq \eps^{-3}$, $\vy^{(i)} \in P$; and therefore, the output of Algorithm~\ref{alg:frank_wolfe} also belongs to $P$. For $i = 0$ this holds by the initialization of $\vy^{(0)}$. For larger values of $i$, this follows by induction because $\vy^{(i)}$ is a convex combination of $\vy^{(i - 1)}$ and the point $\vs^{(i)}$ ($\vy^{(i - 1)}$ belongs to $P$ by the induction hypothesis, and $\vs^{(i)}$ belongs to $P$ by definition).

Our next objective is to lower bound the value of the output point of Algorithm~\ref{alg:frank_wolfe}. For that purpose, it will be useful to define $\bar{F}(\vx) = G(\vx) + 2C(\vx)$ and $H(i) = \bar{F}(\vo) - \bar{F}(\vy^{(i)})$. To get a bound on the value of the output of Algorithm~\ref{alg:frank_wolfe}, we first show that $H(i)$ is small for at least some $i$ value. We do that using the next lemma, which shows that $H(i)$ decreases as a function of $i$ as longs as it is not already small compared to $G(\vy^{i})$. Then, Corollary \ref{cor:low_H_i} guarantees the existence of a good iteration $i^*$.
\begin{lemma} \label{lem:H_decrease}
For every integer $1 \leq i \leq \eps^{-3}$, $H(i - 1) - H(i) \geq \eps^2 \cdot [G(\vo) - 2G(\vy^{(i - 1)})] + 2\eps^2 \cdot [C(\vo) - C(\vy^{(i - 1)})] - 6\eps^4 L D^2 = \eps^2 \cdot [H(i - 1) - G(\vy^{(i - 1)})] - 6\eps^4 L D^2$.
\end{lemma}

\begin{proof}
Observe that
\begin{align*}
	H(i - 1) - H(i)
	={} &
	\bar{F}(\vy^{(i)}) - \bar{F}(\vy^{(i - 1)})
	=
	\bar{F}((1 - \eps^2) \cdot \vy^{(i - 1)} + \eps^2 \cdot \vs^{(i)}) - \bar{F}(\vy^{(i - 1)})\\
	={} &
	\int_0^{\eps^2} \left. \frac{\bar{F}((1 - z) \cdot \vy^{(i - 1)} + z \cdot \vs^{(i)})}{dz}\right|_{z = r} dr\\
	={} &
	\int_0^{\eps^2} \inner{\vs^{(i)} - \vy^{(i - 1)}, \nabla \bar{F}((1 - r) \cdot \vy^{(i - 1)} + r \cdot \vs^{(i)})} dr\\
	\geq{} &
	\int_0^{\eps^2} \left[\inner{\vs^{(i)} - \vy^{(i - 1)}, \nabla \bar{F}(\vy^{(i - 1)})} - 12rLD^2\right] dr\\
	={} &
	\eps^2 \cdot \inner{\vs^{(i)} - \vy^{(i - 1)}, \nabla \bar{F}(\vy^{(i - 1)})} - 6\eps^4LD^2
	\enspace.
\end{align*}
To make the expression on the rightmost side useful, we need to lower bound the inner product $\inner{\vs^{(i)} - \vy^{(i - 1)}, \nabla \bar{F}(\vy^{(i - 1)})}$.
\begin{align*}
	\inner{\vs^{(i)} - \vy^{(i - 1)}, \nabla \bar{F}(\vy^{(i - 1)})}\mspace{-80mu}&\\
={} &
	\inner{\vs^{(i)}, \nabla G(\vy^{(i - 1)}) + 2\nabla C(\vy^{(i - 1)})} - \inner{\vy^{(i - 1)}, \nabla G(\vy^{(i - 1)}) + 2\nabla C(\vy^{(i - 1)})}\\
	\geq{} &
	\inner{\vo, \nabla G(\vy^{(i - 1)}) + 2\nabla C(\vy^{(i - 1)})} - \inner{\vy^{(i - 1)}, \nabla G(\vy^{(i - 1)}) + 2\nabla C(\vy^{(i - 1)})}\\
	={} &
	\inner{\vo - \vy^{(i - 1)}, \nabla G(\vy^{(i - 1)})} + \inner{2(\vo - \vy^{(i - 1)}), \nabla C(\vy^{(i - 1)})}
	\enspace,
\end{align*}
where the inequality holds since $\vs^{(i)}$ is the maximizer found by Algorithm~\ref{alg:frank_wolfe}. Combining the last two inequalities yields
\[
	H(i - 1) - H(i)
	\geq
	\inner{\eps^2 (\vo - \vy^{(i - 1)}), \nabla G(\vy^{(i - 1)})} + \inner{2\eps^2 (\vo - \vy^{(i - 1)}), \nabla C(\vy^{(i - 1)})} - 6\eps^4 L D^2
	\enspace.
\]
Therefore, to complete the proof of the lemma it remains to prove that $\inner{\vo - \vy^{(i - 1)}, \nabla G(\vy^{(i - 1)})} \geq G(\vo) - 2G(\vy^{(i - 1)})$ and $\inner{\vo - \vy^{(i - 1)}, \nabla C(\vy^{(i - 1)})} \geq C(\vo) - C(\vy^{(i - 1)})$.

The inequality $\inner{\vo - \vy^{(i - 1)}, \nabla C(\vy^{(i - 1)})} \geq C(\vo) - C(\vy^{(i - 1)})$ follows immediately from the concavity of $C$. To prove the other inequality, observe that the DR-submodularity of $G$ implies $\inner{\vx^{(1)}, \nabla G(\vx^{(2)})} \geq \inner{\vx^{(1)}, \nabla G(\vx^{(1)} + \vx^{(2)})}$ for every two vectors $\vx^{(1)}$ and $\vx^{(2)}$ that obey $\vx^{(2)}  \in [0, 1]^n$, $\vx^{(1)} + \vx^{2} \in [0, 1]^n$ and $\vx^{(1)}$ is either non-negative or non-positive. This observation implies
{\allowdisplaybreaks\begin{align*}
	&
	\inner{\vo - \vy^{(i - 1)}, \nabla G(\vy^{(i - 1)})}\\
	={} &
	\inner{\vo \vee \vy^{(i - 1)} - \vy^{(i - 1)}, \nabla G(\vy^{(i - 1)})} + \inner{\vo \wedge \vy^{(i - 1)} - \vy^{(i - 1)}, \nabla G(\vy^{(i - 1)})}\\
	\geq{} &
	\int_0^1 \left[\inner{\vo \vee \vy^{(i - 1)} - \vy^{(i - 1)}, \nabla G(\vy^{(i - 1)} + r \cdot (\vo \vee \vy^{(i - 1)} - \vy^{(i - 1)}))} \right.\\&\left.+ \inner{\vo \wedge \vy^{(i - 1)} - \vy^{(i - 1)}, \nabla G(\vy^{(i - 1)} + r \cdot (\vo \wedge \vy^{(i - 1)} - \vy^{(i - 1)}))}\right] dr\\
	={} &
	\int_0^1 \left. \frac{dG(\vy^{(i - 1)} + z \cdot (\vo \vee \vy^{(i - 1)} - \vy^{(i - 1)})) + dG(\vy^{(i - 1)} + z \cdot (\vo \wedge \vy^{(i - 1)} - \vy^{(i - 1)}))}{dz}\right|_{z = r} dr\\
	={} &
	[G(\vo \vee \vy^{(i - 1)}) + G(\vo \wedge \vy^{(i - 1)})] - 2G(\vy^{(i - 1)})
	\geq
	G(\vo) - 2G(\vy^{(i - 1)})
	\enspace,
\end{align*}}
where the last inequality follows from the monotonicity and non-negativity of $G$.
\end{proof}

\begin{corollary} \label{cor:low_H_i}
There is an integer $0 \leq i^* \leq \eps^{-3}$ obeying $H(i^*) \leq G(\vy^{(i^*)}) + \eps \cdot [G(\vo) + 2\eta + 6 L D^2]$.
\end{corollary}

\begin{proof}
Assume towards a contradiction that the corollary does not hold. By Lemma~\ref{lem:H_decrease}, this implies
\[
	H(i - 1) - H(i)
	\geq
	\eps^2 \cdot [H(i - 1) - G(\vy^{(i - 1)})] - 6\eps^4 L D^2
	\geq
	\eps^3 \cdot [G(\vo) + 2\eta] - 6\eps^4 L D^2
\]
for every integer $1 \leq i \leq \eps^{-3}$. Adding up this inequality for all these values of $i$, we get
\[
	H(0) - H(\eps^{-3})
	\geq
	[G(\vo) + 2\eta] - 6\eps L D^2
	\enspace.
\]
However, by the choice of $\vy^{(0)}$, $H(0)$ is not too large. Specifically,
\[
	H(0)
	=
	[G(\vo) + 2C(\vo)] - [G(\vy^{(0)}) + 2C(\vy^{(0)})]
	\leq
	[G(\vo) + 2C(\vo)] - \{0 + 2[C(\vo) - \eta]\}
	=
	G(\vo) + 2\eta
	\mspace{-1mu}\enspace,
\]
where the inequality holds since $\vo$ is one possible candidate to be $\vy^{(0)}$ and $G$ is non-negative. Therefore, we get
\[
	[G(\vo) + 2\eta] - H(\eps^{-3})
	\geq
	[G(\vo) + 2\eta] - 6\eps L D^2
	\enspace,
\]
which by the non-negativity of $G$ and $\eta$ implies
\[
	H(\eps^{-3})
	\leq
	6\eps L D^2
	\leq
	G(\eps^{-3}) + \eps \cdot [G(\vo) + 2\eta + 6 L D^2]
\]
(which contradicts our assumption).
\end{proof}

We are now ready to summarize the properties of Algorithm~\ref{alg:frank_wolfe} in a theorem.
\begin{theorem} \label{thm:frank_wolfe_analysis}
Let $S_1$ be the time it takes to find a point in $P$ maximizing a given linear function and $S_2$ be the time it takes to find a point in $P$ maximizing $C(\cdot)$ up to an error of $\eta$, then Algorithm~\ref{alg:frank_wolfe} runs in $O(\eps^{-3} \cdot (n + S_1) + S_2)$ time, makes $ O(1 / \eps^{3})$ gradient oracle calls, and outputs a vector $\vy$ such that
\[
	F(\vy)
	\geq
	\tfrac{1}{2}(1 - \eps) \cdot G(\vo) + C(\vo) - \eps \cdot O(\eta + L D^2)
	\enspace.
\]
\end{theorem}

\begin{proof}

We begin the proof by analyzing the time and oracle complexities of Algorithm~\ref{alg:frank_wolfe}. Every iteration of the loop of Algorithm~\ref{alg:frank_wolfe} takes $O(n + S_1)$ time. As there are $\eps^{-3}$ such iterations, the entire algorithm runs in $O(\eps^{-3} (n + S_1) + S_2)$ time. Also note that each iteration of the loop requires 2 calls to the gradient oracles (a single call to the oracle corresponding to $G$, and a single call to the oracle corresponding to $C$), so the overall oracle complexity of the algorithm is $O(1/\eps^{3})$.

Consider now iteration $i^*$, whose existence is guaranteed by Corollary~\ref{cor:low_H_i}. Then,
\begin{align*}
	&
	H(i^*) \leq G(\vy^{(i^*)}) + \eps \cdot [G(\vo) + 2\eta + 6 L D^2]\\
	\implies{} &
	[G(\vo) + 2C(\vo)] - [G(\vy^{(i^*)}) + 2C(\vy^{(i^*)})] \leq G(\vy^{(i^*)})
	+ \eps \cdot [G(\vo) + 2\eta + 6 L D^2]\\
	\implies{} &
	(1 - \eps) \cdot G(\vo) + 2C(\vo) \leq 2[G(\vy^{(i^*)}) + C(\vy^{(i^*)})] + \eps \cdot [2\eta + 6 L D^2]\\
	\implies{} &
	F(\vy^{(i^*)}) \geq \tfrac{1}{2}(1 - \eps) \cdot G(\vo) + C(\vo) - \eps \cdot [\eta + 3 L D^2]
	\enspace.
\end{align*}
The theorem now follows since the output of Algorithm~\ref{alg:frank_wolfe} is at least as good as $\vy^{(i^*)}$.
\end{proof}

\subsection{Non-oblivious Frank-Wolfe Algorithm} \label{section:nonoblivious_fw}

As mentioned in the beginning of Section~\ref{section:grad_combfw}, our objective in this section is to present another algorithm that has $\beta = 1$ (\ie, it maximizes $C$ ``exactly'' in some sense). In Section~\ref{section:grad_combfw}, we presented Algorithm \ref{alg:frank_wolfe}, which achieves this goal with $\alpha = 1/2$. The algorithm we present in the current section achieves the same goal with an improved value of $1 - 1/e$ for $\alpha$. However, the improvement is obtained at the cost of requiring the function $C$ to be non-negative (which was not required in Section~\ref{section:grad_combfw}). Additionally, like in the previous section, we assume here that $G$ is a monotone and non-negative function (in addition to its other properties).

The algorithm we study in this section is a non-oblivious variant of the Frank-Wolfe algorithm, appearing as Algorithm~\ref{alg:non_ob_fw}, which takes a quality control parameter $\eps \in (0, 1/4)$ as input. As usual, we assume without loss of generality that $\eps^{-1}$ is an integer. Algorithm~\ref{alg:non_ob_fw} also employs the non-negative auxiliary function:
\[
	\bar{G}(\vx)
	=
	\eps \cdot \sum_{j = 1}^{\eps^{-1}} \frac{e^{\eps j} \cdot G(\eps j \cdot \vx)}{\eps j}
	.
\]
This function is inspired by the non-oblivious objective function used by~\citet{filmus2012matroid}.

\begin{algorithm}[H] 
\caption{\textsc{Non-Oblivious Frank-Wolfe} $(\eps)$} \label{alg:non_ob_fw}
\begin{algorithmic}
\STATE Let $\vy^{(0)}$ be an arbitrary vector in $P$, and let $\beta(\eps) \gets e(1 - \ln \eps)$.
\FOR{$i = 0$ \textbf{to} $\lceil e^{-1} \cdot \beta(\eps) / \eps^2 \rceil$}
\STATE $\vs^{(i)} \gets \arg \max_{\vx \in P} \inner{ e^{-1} \! \cdot \! \nabla \bar{G}(\vy^{(i)}) + \nabla C(\vy^{(i)}) , \vx}$
\STATE $\vy^{(i + 1)} \gets (1 - \eps) \cdot \vy^{(i)} + \eps \cdot \vs^{(i)}$
\ENDFOR
\RETURN{the vector maximizing $F$ among $\{\vy^{(0)},\dotsc,\allowbreak \vy^{(\lceil e^{-1} \cdot \frac{\beta(\eps)}{\eps^2} \rceil)}$\}}
\end{algorithmic}
\end{algorithm}

Note that any call to the gradient oracle of $\bar{G}$ can be simulated using $\eps^{-1}$ calls to the gradient oracle of $G$.
The properties of Algorithm~\ref{alg:non_ob_fw} are stated in Theorem~\ref{thm:non_ob_fw}. 
We begin the analysis of Algorithm~\ref{alg:non_ob_fw} by observing that for every integer $0 \leq i \leq \lceil e^{-1} \cdot \beta(\eps)/\eps^2 \rceil$, $\vy^{(i)} \in P$; and therefore, the output of Algorithm~\ref{alg:non_ob_fw} also belongs to $P$. The proof for this is identical to the corresponding proof in Section~\ref{section:grad_combfw}.

Let us now analyze the auxiliary function $\bar{G}$. Specifically, we need to show that $\bar{G}$ is almost as smooth as the original function $G$, and that for every given vector $\vx \in [0, 1]$, the value of $\bar{G}(\vx)$ is not very large compared to $G(\vx)$. We mention these as observations below.

\begin{observation}\label{obs:G_smooth}
The auxiliary function $\bar{G}$ is $eL$-smooth.
\end{observation}
\begin{proof}
For every two vectors $\vx, \vy \in [0, 1]^n$,
\begin{align*}
	\|\nabla& \bar{G}(\vx) - \nabla \bar{G}(\vy)\|_2
	=
	\left\|\eps \cdot \sum_{j = 1}^{\eps^{-1}} \frac{e^{\eps j} \cdot [\nabla G(\eps j \cdot \vx) - \nabla G(\eps j \cdot \vy)]}{\eps j}\right\|_2\\
	\leq{} &
	\eps \cdot \sum_{j = 1}^{\eps^{-1}} \frac{e^{\eps j} \cdot \|\nabla G(\eps j \cdot \vx) - \nabla G(\eps j \cdot \vy)\|_2}{\eps j}
	\leq
	\eps \cdot \sum_{j = 1}^{\eps^{-1}} \frac{e \cdot L\eps j \cdot \|\vx - \vy\|_2}{\eps j}
	=
	eL \cdot \|\vx - \vy\|_2
	\enspace.
	\qedhere
\end{align*}
\end{proof}

\begin{observation} \label{obs:G_bar_range}
For every vector $\vx \in [0, 1]^n$, $\bar{G}(\vx) \leq \beta(\eps) \cdot G(\vx)$. 
\end{observation}
\begin{proof}
Note that, by the monotonicity of $G$,
\begin{align*}
	\bar{G}(\vx)
	={} &
	\eps \cdot \sum_{j = 1}^{\eps^{-1}} \frac{e^{\eps j} \cdot G(\eps j \cdot \vx)}{\eps j}
	\leq
	\eps \cdot \sum_{j = 1}^{\eps^{-1}} \frac{e \cdot G(\vx)}{\eps j}\\
	={} &
	e \cdot G(\vx) \cdot \sum_{j = 1}^{\eps^{-1}} \frac{1}{j}
	\leq
	e(1 - \ln \eps) \cdot G(\vx)
	=
	\beta(\eps) \cdot G(\vx)
	\enspace,
\end{align*}
where the last inequality holds since
\[
	\sum_{j = 1}^{\eps^{-1}} \frac{1}{j}
	\leq
	1 + \int_1^{\eps^{-1}} \frac{dx}{x}
	=
	1 + [\ln x]_1^{\eps^{-1}}
	=
	1 - \ln \eps
	\enspace.
	\qedhere
\]
\end{proof}

We now define $\bar{F}(\vx) = e^{-1} \cdot \bar{G}(\vx) + C(\vx)$. To get a bound on the value of the output of Algorithm~\ref{alg:non_ob_fw}, we need to show that there exists at least one value of $i$ for which $\bar{F}(\vy^{(i + 1)})$ is not much larger than $\bar{F}(\vy^{(i)})$, which intuitively means that $\vy^{(i)}$ is roughly a local maximum with respect to $\bar{F}$. The following lemma proves that such an $i$ value indeed exists.
\begin{lemma} \label{lem:local_optimum_exists}
There is an integer $0 \leq i^* < \lceil e^{-1} \cdot \beta(\eps)/\eps^2 \rceil$ such that $\bar{F}(\vy^{(i^* + 1)}) - \bar{F}(\vy^{(i^*)}) \leq \eps^2 \cdot F(\vo)$.
\end{lemma}
\begin{proof}
Assume towards a contradiction that the lemma does not hold. This implies
\begin{align*}
	\bar{F}(\vy^{(\lceil \beta(\eps)/\eps^2 \rceil)}) - \bar{F}(\vy^{(0)})
	={} &
	\sum_{i = 0}^{\lceil \beta(\eps)/\eps^2 \rceil - 1} [\bar{F}(\vy^{(i + 1)}) - \bar{F}(\vy^{(i)})]\\
	>{} &
	\lceil e^{-1} \cdot \beta(\eps)/\eps^2 \rceil \cdot \eps^{2} \cdot F(\vo)
	\geq
	e^{-1} \cdot \beta(\eps) \cdot F(\vo)
	\enspace.
\end{align*}
Furthermore, using the non-negativity of $\bar{F}$, the last inequality implies
\begin{align*}
	e^{-1} \cdot \bar{G}(\vy^{(\lceil \beta(\eps)/\eps^2 \rceil)}) + C(\vy^{(\lceil \beta(\eps)/\eps^2 \rceil)})
	={} &
	\bar{F}(\vy^{(\lceil \beta(\eps)/\eps^2 \rceil)})\\
	\geq{} &
	\bar{F}(\vy^{(\lceil \beta(\eps)/\eps^2 \rceil)}) - \bar{F}(\vy^{(0)})
	>
	e^{-1} \cdot \beta(\eps) \cdot F(\vo)
	\enspace.
\end{align*}
Let us now upper bound the two terms on the leftmost side of the above inequality. By Observation~\ref{obs:G_bar_range}, we can upper bound $\bar{G}(\vy^{(\lceil \beta(\eps)/\eps^2 \rceil)})$ by $\beta(\eps) \cdot G(\vy^{(\lceil \beta(\eps)/\eps^2 \rceil)})$. Additionally, since $\beta(\eps) \geq e$ because $\eps < 1$, we can upper bound $C(\vy^{(\lceil \beta(\eps)/\eps^2 \rceil)})$ by $e^{-1} \cdot \beta(\eps) \cdot C(\vy^{(\lceil \beta(\eps)/\eps^2 \rceil)})$. Plugging both these upper bounds into the previous inequalities and dividing by $e^{-1} \cdot \beta(\eps)$ gives
\[
	F(\vy^{(\lceil \beta(\eps)/\eps^2 \rceil)})
	=
	G(\vy^{(\lceil \beta(\eps)/\eps^2 \rceil)}) + C(\vy^{(\lceil \beta(\eps)/\eps^2 \rceil)})
	>
	F(\vo)
	\enspace,
\]
which contradicts the definition of $\vo$, and thus, completes the proof of the lemma.
\end{proof}

The preceding lemma shows that $\vy^{(i^*)}$ is roughly a local optimum in terms of $\bar{F}$. We now show an upper bound on $\inner{(\vo - \vy^{(i^*)}), \nabla \bar{G}(\vy^{(i^*)})}$ that is implied by this property of $\vy^{(i^*)}$.
\begin{lemma} \label{lem:H_decrease_non_oblivious}
$\inner{e^{-1} \cdot (\vo - \vy^{(i^*)}), \nabla \bar{G}(\vy^{(i^*)})} \leq \eps \cdot F(\vo) + [C(\vy^{(i^*)}) - C(\vo)] + 4\eps LD^2$.
\end{lemma}
\begin{proof}
Observe that
\begin{align*}
	&
	\bar{F}(\vy^{(i^* + 1)}) - \bar{F}(\vy^{(i^*)})
	=
	\bar{F}((1 - \eps) \cdot \vy^{(i^*)} + \eps \cdot \vs^{(i^*)}) - \bar{F}(\vy^{(i^*)})\\
	={} &
	\int_0^\eps \left. \frac{\bar{F}((1 - z) \cdot \vy^{(i^*)} + z \cdot \vs^{(i^*)})}{dz}\right|_{z = r} dr\\
	={} &
	\int_0^\eps \inner{\vs^{(i^*)} - \vy^{(i^*)}, \nabla \bar{F}((1 - r) \cdot \vy^{(i^*)} + r \cdot \vs^{(i^* + 1)})} dr\\
	\geq{} &
	\int_0^\eps \left[\inner{\vs^{(i^*)} - \vy^{(i^*)}, \nabla \bar{F}(\vy^{(i^*)})} - 8rLD^2\right] dr
	=
	\inner{\eps \cdot (\vs^{(i^*)} - \vy^{(i^*)}), \nabla \bar{F}(\vy^{(i^*)})} - 4\eps^2LD^2
	\enspace.
\end{align*}
To make the expression in the rightmost side useful, we need to lower bound the expression $\inner{\vs^{(i^*)} - \vy^{(i^*)}, \nabla \bar{F}(\vy^{(i^*)})}$.
\begin{align*}
	(\vs^{(i^*)} -{}& \vy^{(i^*)}) \cdot \nabla \bar{F}(\vy^{(i^*)})\\
	={} &
	\inner{\vs^{(i^*)}, e^{-1} \cdot \nabla \bar{G}(\vy^{(i^*)}) + \nabla C(\vy^{(i^*)})} - \inner{\vy^{(i^*)}, e^{-1} \cdot \nabla \bar{G}(\vy^{(i^*)}) + \nabla C(\vy^{(i^*)})}\\
	\geq{} &
	\inner{\vo, e^{-1} \cdot \nabla \bar{G}(\vy^{(i^*)}) + \nabla C(\vy^{(i^*)})} - \inner{\vy^{(i^*)}, e^{-1} \cdot \nabla \bar{G}(\vy^{(i^*)}) + \nabla C(\vy^{(i^*)})}\\
	={} &
	\inner{e^{-1} \cdot (\vo - \vy^{(i^*)}), \nabla \bar{G}(\vy^{(i^*)})} + \inner{\vo - \vy^{(i^*)}, \nabla C(\vy^{(i^*)})}
	\enspace,
\end{align*}
where the inequality holds since $\vs^{(i)}$ is the maximizer found by Algorithm~\ref{alg:non_ob_fw}. Combining the last two inequalities and the guarantee on $i^*$ from Lemma~\ref{lem:local_optimum_exists}, we now get
\begin{align*}
	\eps^2 \cdot F(\vo)
	\geq{} &
	\bar{F}(\vy^{(i^* + 1)}) - \bar{F}(\vy^{(i^*)})
	\geq
	\inner{\eps \cdot (\vs^{(i^*)} - \vy^{(i^*)}), \nabla \bar{F}(\vy^{(i^*)})} - 4\eps^2LD^2\\
	\geq{} &
	\eps \cdot \left[\inner{e^{-1} \cdot (\vo - \vy^{(i^*)}), \nabla \bar{G}(\vy^{(i^*)})} + \inner{\vo - \vy^{(i^*)}, \nabla C(\vy^{(i^*)})}\right] - 4\eps^2LD^2\\
	\geq{} &
	\eps \cdot \left[\inner{e^{-1} \cdot (\vo - \vy^{(i^*)}), \nabla \bar{G}(\vy^{(i^*)})} + C(\vo) - C(\vy^{(i - 1)})\right] - 4\eps^2LD^2
	\enspace,
\end{align*}
where the last inequality follows from the concavity of $C$. The lemma now follows by rearranging the last inequality (and dividing it by $\eps$).
\end{proof}

We now wish to lower bound $\inner{(\vo - \vy^{(i^*)}), \nabla \bar{G}(\vy^{(i^*)})}$ and we do so by separately analyzing $\inner{\vo, \nabla \bar{G}(\vy^{(i^*)})}$ and $\inner{\vy^{(i^*)}, \nabla \bar{G}(\vy^{(i^*)})}$ in the following two lemmas.

\begin{lemma} \label{lem:o_gradient_bound}
$\inner{\vo, \nabla \bar{G}(\vy^{(i^*)})} \geq (e - 1) \cdot G(\vo) - \eps \cdot \sum_{j = 1}^{\eps^{-1}} e^{\eps j} \cdot G(\eps j \cdot \vy^{(i^*)})$.
\end{lemma}
\begin{proof}
For brevity, we use in this proof the shorthand $\vu^{(j)} = \eps j \cdot \vy^{(i^*)}$. Then, for every integer $0 \leq j \leq \eps^{-1}$, we get (in these calculations, the expression $\nabla G(\eps j \cdot \vy^{(i^*)})$ represents the gradient of the function $G(\eps j \cdot x)$ at the point $x = \vy^{(i^*)}$).
\begin{align*}
	\frac{\inner{\vo, \nabla G(\eps j \cdot \vy^{(i^*)})}}{\eps j}
	\geq{} \mspace{-61mu}&\mspace{61mu}
	\frac{\inner{\vo \vee \vu^{(j)} - \vu^{(j)}, \nabla G(\eps j \cdot \vy^{(i^*)})}}{\eps j}
	=
	\left. \frac{dG(\vu^{(j)} + z(\vo \vee \vu^{(j)} - \vu^{(j)}))}{dz} \right|_{z = 0}\\
	\geq{} &
	\int_0^1 \left. \frac{dG(\vu^{(j)} + z \cdot (\vo \vee \vu^{(j)} - \vu^{(j)}))}{dz}\right|_{z = r} dr
	=
	\left[ G(\vu^{(j)} + z \cdot (\vo \vee \vu^{(j)} - \vu^{(j)}))\right]_0^1\\
	={} &
	G(\vo \vee \vu^{(uj)}) - G(\eps j \cdot \vy^{(i^*)})
	\geq
	G(\vo) - G(\eps j \cdot \vy^{(i^*)})
	\enspace,
\end{align*}
where the first and last inequalities follow from the monotonicity of $G$ and the second inequality follows from the DR-submodularity of $G$. Using the last inequality, we now get
\begin{align*}
	\inner{\vo, \nabla \bar{G}(\vy^{(i^*)})}
	={} &
	\eps \cdot \sum_{j = 1}^{\eps^{-1}} \frac{\inner{e^{\eps j} \cdot \vo, \nabla G(\eps j \cdot \vy^{(i^*)})}}{\eps j}
	\geq
	\eps \cdot \sum_{j = 1}^{\eps^{-1}} e^{\eps j} \cdot [G(\vo) - G(\eps j \cdot \vy^{(i^*)})]\\
	\geq{} &
	(e - 1) \cdot G(\vo) - \eps \cdot \sum_{j = 1}^{\eps^{-1}} e^{\eps j} \cdot G(\eps j \cdot \vy^{(i^*)})
	\enspace,
\end{align*}
where the second inequality holds since the fact that $\sum_{j = 1}^{\eps^{-1}} e^{\eps j}$ is a geometric sum implies
\[
	\eps \cdot \sum_{j = 1}^{\eps^{-1}} e^{\eps j}
	=
	\eps \cdot e^{\eps} \cdot \frac{e - 1}{e^{\eps} - 1}
	=
	\frac{\eps}{1 - e^{-\eps}} \cdot (e - 1)
	\geq
	e - 1
	\enspace.
	\qedhere
\]
\end{proof}

\begin{lemma} \label{lem:y_gradient_bound}
$\inner{\vy^{(i^*)}, \nabla \bar{G}(\vy^{(i^*)})} \leq (e + 6\eps \ln \eps^{-1}) \cdot G(\vy^{(i^*)}) - \eps  \cdot \sum_{j = 1}^{\eps^{-1}} e^{\eps j} \cdot G(\eps j \cdot \vy^{(i^*)})$.
\end{lemma}
\begin{proof}
In the following calculation, the expression $\nabla G(\eps j \cdot \vy^{(i^*)})$ again represents the gradient of the function $G(\eps j \cdot x)$ at the point $x = \vy^{(i^*)}$). Thus,
\begin{align} \label{eq:vy_upperbound_step_1}
	\inner{\vy^{(i^*)}, \nabla \bar{G}(\vy^{(i^*)})}
	={} &
	\eps \cdot \sum_{j = 1}^{\eps^{-1}} \frac{\inner{e^{\eps j} \cdot \vy^{(i^*)}, \nabla G(\eps j \cdot \vy^{(i^*)})}}{\eps j}
	=
	\eps \cdot \sum_{j = 1}^{\eps^{-1}} e^{\eps j} \cdot \left. \frac{dG(z \cdot \vy^{(i^*)})}{dz} \right|_{z = \eps j}\\ \nonumber
	\leq{} &
	\eps \cdot \sum_{h = 1}^{\eps^{-1}} \left[\left. \frac{dG(z \cdot \vy^{(i^*)})}{dz} \right|_{z = \eps h} + \eps \cdot e^{\eps h} \cdot \sum_{j = h}^{\eps^{-1}} \left. \frac{dG(z \cdot \vy^{(i^*)})}{dz} \right|_{z = \eps j}\right]
	\enspace,
\end{align}
where the inequality holds because
\[
	\eps \cdot \sum_{h = 1}^{j} e^{\eps h}
	=
	\eps \cdot e^{\eps} \cdot \frac{e^{\eps j} - 1}{e^\eps - 1}
	=
	\frac{\eps}{1 - e^{-\eps}} \cdot (e^{\eps j} - 1)
	\geq
	e^{\eps j} - 1
	\enspace.
\]

We now observe also that the DR-submodularity of $G$ implies, for every integer $1 \leq h \leq \eps^{-1}$, that
\begin{align*}
	\eps \cdot \sum_{j = h}^{\eps^{-1}} \left. \frac{dG(z \cdot \vy^{(i^*)})}{dz} \right|_{z = \eps j}
	\leq{} &
	h^{-1} \cdot \int_0^{\eps h} \left. \frac{dG(z \cdot \vy^{(i^*)})}{dz} \right|_{z = x} dx + \int_{\eps h}^1 \left. \frac{dG(z \cdot \vy^{(i^*)})}{dz} \right|_{z = x} dx\\
	={} &
	h^{-1} \cdot [G(x \cdot \vy^{(i^*)})]_{0}^{\eps h} + [G(x \cdot \vy^{(i^*)})]_{\eps h}^1\\
	={} &
	h^{-1} \cdot G(\eps h \cdot \vy^{(i^*)}) - h^{-1} \cdot G(\vzero) + G(\vy^{(i^*)}) - G(\eps h \cdot \vy^{(i^*)})\\
	\leq{} &
	G(\vy^{(i^*)}) - (1 - h^{-1}) \cdot G(\eps h \cdot \vy^{(i^*)})
	\enspace,
\end{align*}
where the last inequality follows from the non-negativity of $G$. Plugging the last inequality now into Inequality~\eqref{eq:vy_upperbound_step_1} yields
{\allowdisplaybreaks\begin{align} \label{eq:before_coefficient_analysis}
	\inner{\vy^{(i^*)}, \nabla \bar{G}(\vy^{(i^*)})}\mspace{-51mu}{}&\mspace{51mu}
	\leq
	G(\vy^{(i^*)}) + \eps \cdot \sum_{j = 1}^{\eps^{-1}} \left[e^{\eps j} \cdot \{G(\vy^{(i^*)}) - (1 - j^{-1}) \cdot G(\eps j \cdot \vy^{(i^*)})\}\right]\\ \nonumber
	={} &
	G(\vy^{(i^*)}) + \eps \cdot \sum_{j = 1}^{\eps^{-1}} e^{\eps j} \cdot G(\vy^{(i^*)}) - \eps \cdot \sum_{j = 1}^{\eps^{-1}} (1 - j^{-1}) \cdot e^{\eps j} \cdot G(\eps j \cdot \vy^{(i^*)})\\ \nonumber
	\leq{} &
	G(\vy^{(i^*)}) +\eps \cdot \sum_{j = 1}^{\eps^{-1}} e^{\eps j} \cdot G(\vy^{(i^*)}) - \eps^2 \cdot \left(\sum_{j = 1}^{\eps^{-1}} (1 - j^{-1})\right) \cdot \left(\sum_{j = 1}^{\eps^{-1}} e^{\eps j} \cdot G(\eps j \cdot \vy^{(i^*)})\right)\\ \nonumber
	\leq{} &
	G(\vy^{(i^*)}) + \eps \cdot \sum_{j = 1}^{\eps^{-1}} e^{\eps j} \cdot G(\vy^{(i^*)}) - \eps(1 + 2\eps \ln \eps)  \cdot \sum_{j = 1}^{\eps^{-1}} e^{\eps j} \cdot G(\eps j \cdot \vy^{(i^*)})\\ \nonumber
	\leq{} &
	G(\vy^{(i^*)}) + \left(\eps - 2\eps^2 \ln \eps \right) \cdot \sum_{j = 1}^{\eps^{-1}} e^{\eps j} \cdot G(\vy^{(i^*)}) - \eps  \cdot \sum_{j = 1}^{\eps^{-1}} e^{\eps j} \cdot G(\eps j \cdot \vy^{(i^*)})
	\enspace,
\end{align}}
where the last inequality follows from the monotonicity of $G$ (since $2\eps \ln \eps < 0$) and the second inequality follows from Chebyshev's sum inequality because the monotonicity of $G$ implies that both $1 - j^{-1}$ and $e^{\eps j} \cdot G(\eps j \cdot \vy^{(i^*)})$ are non-decreasing functions of $j$. Additionally, the penultimate inequality follows due to the following calculation, which holds for every $\eps \in (0, 1/3)$.
\[
	\eps \cdot \sum_{j = 1}^{\eps^{-1}} (1 - j^{-1})
	\geq
	\eps \cdot \int_1^{\eps^{-1}} (1 - x^{-1}) dx
	=
	\eps [x - \ln x]_1^{\eps^{-1}}
	=
	1 - \eps + \eps \ln \eps
	\geq
	1 + 2\eps \ln \eps
	\enspace.
\]

To complete the proof of the lemma, it remains to show that the coefficient of $G(\vy^{(i^*)})$ in the rightmost side of Inequality~\eqref{eq:before_coefficient_analysis} is upper bounded by $e + 6\eps \ln \eps^{-1}$. To do that, we note that this coefficient is
\begin{align*}
	1 + (\eps - 2\eps^2 \ln \eps) &{}\cdot \sum_{j = 1}^{\eps^{-1}} e^{\eps j}
	=
	1 + (\eps - 2\eps^2 \ln \eps) \cdot e^{\eps} \cdot \frac{e - 1}{e^\eps - 1}
	=
	1 + \frac{\eps (1 - 2\eps \ln \eps)}{1 - e^{-\eps}} \cdot (e - 1)\\
	\leq{} &
	1 + \frac{1 - 2\eps \ln \eps}{1 - \eps/2} \cdot (e - 1)
	\leq
	1 + (1 - 2\eps \ln \eps)(1 + \eps)(e - 1)\\
	={} &
	e + \eps(e - 1) \cdot [1 - 2(1 + \eps) \ln \eps]
	\leq
	e + 6\eps \ln \eps^{-1}
	\enspace,
\end{align*}
where the first inequality and inequalities hold because $e^{-\eps} \leq 1 - \eps + \eps^2/2$ and $(1 - \eps/2)^{-1} \leq (1 + \eps)$ for $\eps \in (0, 1)$, and the last inequality holds for $\eps \in (0, 1/4)$.
\end{proof}

Combining the above lemmas then leads to the following corollary.
\begin{corollary} \label{cor:non_oblivious_guarantee}
$G(\vy^{(i^*)}) + C(\vy^{(i^*)}) \geq (1 - 1/e - 4\eps \ln \eps^{-1}) \cdot G(\vo) + (1 - 4\eps \ln \eps^{-1}) \cdot C(\vo) - 4\eps LD^2$.
\end{corollary}
\begin{proof}
Observe that
\begin{align*}
	(1 - 1/e) \cdot G(\vo) - (1 &{}+ 3\eps \ln \eps^{-1}) \cdot G(\vy^{(i^*)})\\
	\leq{} &
	e^{-1} \cdot \{[(e - 1) \cdot G(\vo) - \eps \cdot \sum_{j = 1}^{\eps^{-1}} e^{\eps j} \cdot G(j \cdot \vy^{(i^*)})] \\&\mspace{100mu}- [(e + 6\eps \ln \eps^{-1}) \cdot G(\vy^{(i^*)}) - \eps  \cdot \sum_{j = 1}^{\eps^{-1}} e^{\eps j} \cdot G(\eps j \cdot \vy^{(i^*)})]\}\\
	\leq{} &
	e^{-1} \cdot \inner{\vo - \vy^{(i^*)}, \nabla \bar{G}(\vy^{(i^*)})}
	\leq
	\eps \cdot F(\vo) + [C(\vy^{(i^*)}) - C(\vo)] + 4\eps LD^2
	\enspace,
\end{align*}
where the second inequality follows from Lemmata~\ref{lem:o_gradient_bound} and~\ref{lem:y_gradient_bound}, and the last inequality follows from Lemma~\ref{lem:H_decrease_non_oblivious}. Rearranging the above inequality now yields
\begin{align*}
	G(\vy^{(i^*)}) + C(\vy^{(i^*)})
	\geq{} &
	(1 - 1/e - \eps) \cdot G(\vo) + (1 - \eps) \cdot C(\vo) - 3\eps \ln \eps^{-1} \cdot G(\vy^{(i^*)}) - 4\eps LD^2\\
	\geq{} &
	(1 - 1/e - 4\eps \ln \eps^{-1}) \cdot G(\vo) + (1 - 4\eps \ln \eps^{-1}) \cdot C(\vo) - 4\eps LD^2
	\enspace,
\end{align*}
where the second inequality holds since $G(\vo) + C(\vo) \geq G(\vy^{(i^*)}) + C(\vy^{(i^*)}) \geq G(\vy^{(i^*)})$ due to the choice of the vector $\vo$ and the non-negativity of $C$.
\end{proof}

We are now ready to state and prove Theorem~\ref{thm:non_ob_fw}.

\begin{restatable}{theorem}{thmNonObFw} \label{thm:non_ob_fw}
Let $S$ be the time it takes to find a point in $P$ maximizing a given linear function, then Algorithm~\ref{alg:non_ob_fw} runs in $O(\eps^{-2}(n/\eps + S)\ln \eps^{-1})$ time, makes $ O(\eps^{-3} \ln \eps^{-1})$ gradient oracle calls, and outputs a vector $\vy$ such that:
\begin{equation*}
F(\vy)
	\geq {}
	(1 - 1/e - 4\eps \ln \eps^{-1}) \cdot G(\vo)
	  + (1 - 4\eps \ln \eps^{-1}) \cdot C(\vo) - 4\eps LD^2
	\enspace.
\end{equation*}
\end{restatable}

\begin{proof}
Once more, we begin by analyzing the time and oracle complexities of the algorithm. Every iteration of the loop of Algorithm~\ref{alg:non_ob_fw} takes $O(n/\eps + S)$ time. As there are $\lceil e^{-1} \cdot \beta(\eps) / \eps^{-2} \rceil = O(\eps^{-2} \ln \eps^{-1})$ such iterations, the entire algorithm runs in $O(\eps^{-2}(n/\eps + S)\ln \eps^{-1})$ time. Furthermore, as each iteration of the loop requires $O(\eps^{-1})$ oracle calls (a single call to the oracle corresponding to $C$, and $\eps^{-1}$ calls to the oracle corresponding to $G$), the overall oracle complexity is $ O(\eps^{-3} \ln \eps^{-1})$.

The theorem now follows since the value of the output of Algorithm~\ref{alg:non_ob_fw} is at least as good as $F(\vy^{(i^*)})$, and Corollary~\ref{cor:non_oblivious_guarantee} guarantees that
$
	F(\vy^{(i^*)})
	\geq
	(1 - 1/e - 4\eps \ln \eps^{-1}) \cdot G(\vo) + (1 - 4\eps \ln \eps^{-1}) \cdot C(\vo) - 4\eps LD^2
	.
$
\end{proof}

\section{Experiments}\label{sec:experiments}
In this section we describe some experiments pertaining to our algorithms for maximizing DR-submodular + concave functions. All experiments are done on a 2015 Apple MacBook Pro with a quad-core 2.2 GHz i7 processor and 16 GB of RAM. 

\subsection{Interpolating Between Constrasting Objectives}\label{experiments:interpolation}

We use our algorithms for maximizing the sum of a DR-submodular function and a concave function to provide a way to achieve a trade-off between different objectives. For example, given a ground set $X$ and a DPP supported on the power set $2^X$, the maximum a posteriori (MAP) of the DPP corresponds to picking the most likely (diverse) set of elements according to the DPP. On the other hand, concave functions can be used to encourage points being closer together and clustered.

Finding the MAP of a DPP is an NP-hard problem. However, continuous approaches employing the multilinear extension \citep{maximizing2011calinescu} or the softmax extension \citep{bian2017continuous, gillenwater2012near} provide strong approximation guarantees for it. The softmax approach is usually preferred as it has a closed form solution which is easier to work with. Now, suppose that $|X| = n$, and let $\mathbf{L}$ be the $n \times n$ kernel of the DPP and $\mathbf{I}$ be the $n \times n$ identity matrix, then $G(\vx) = \log \det[ \diag (\vx) \mathbf{(L-I) + I}]$ is the softmax extension for $\vx \in [0,1]^n$. Here, $ \diag (\vx)$ corresponds to a diagonal matrix with the entries of $\vx$ along its diagonal.

Observe now that, given a vector $\vx \in [0,1]^n$, $x_i$ can be thought of as the likelihood of picking element $i$. Moreover, $\mathbf{L}_{ij}$ captures the similarity between elements $i$ and $j$. Hence, our choice for a concave function which promotes similarity among elements is $C(\vx) = \sum_{i,j} \mathbf{L}_{ij}(1 - (x_i - x_j)^2)$. The rationale behind this is as follows. For a particular pair of elements $i$ and $j$, if $\mathbf{L}_{ij}$ is large, that means that $i$ and $j$ are similar, so we would want $C$ to be larger when $\mathbf{L}_{ij}$ is high, provided that we are indeed \emph{picking both} $i$ and $j$ (i.e., provided that $(x_i - x_j)^2$ is small). One can verify that the function $C(\vx)$ is indeed concave as its Hessian is negative semidefinite.

In our first experiment we fix the ground set to be the set of $20 \times 20 = 400$ points evenly spaced in $[0,1] \times [0,1] \subset \mathbb{R}^2$. We also choose $\mathbf{L}$ to be the Gaussian kernel $\mathbf{L}_{ij} = \exp(-d(i,j)^2/2 \sigma^2)$, where $d(i,j)$ is the Euclidean distance between points $i$ and $j$, and $\sigma = 0.04$.
But to define points $i$ and $j$, we need some ordering of the 400 points on the plane. So, we call point 0 to be at the origin and the index increases till we reach point 19 at $(1,0)$. Then we have point 20 at $ (0,\nicefrac{1}{20})$ and so on till we finally reach point 399 at $(1,1)$.
Given the functions $G$ and $C$ defined above, we optimize in this experiment a combined objective formally specified by $F = \lambda G + (1-\lambda) C$, where $\lambda \in [0,1]$ is a control parameter that can be used to balance the contrasting objectives represented by $G$ and $C$. For example, setting $\lambda = 1$ produces the (spread out) pure DPP MAP solution, setting $\lambda = 0$ produces the (clustered) pure concave solution and $\lambda = 0.5$ produces a solution that takes both constraints into consideration to some extent. It is important to note, however, that the effect of changing $\lambda$ on the importance that each type of constraint gets is not necessarily linear---although it becomes linear when the ranges of $G$ and $C$ happen to be similar.

\begin{figure}[t!]
  \centering
  \includegraphics[width=0.9\columnwidth]{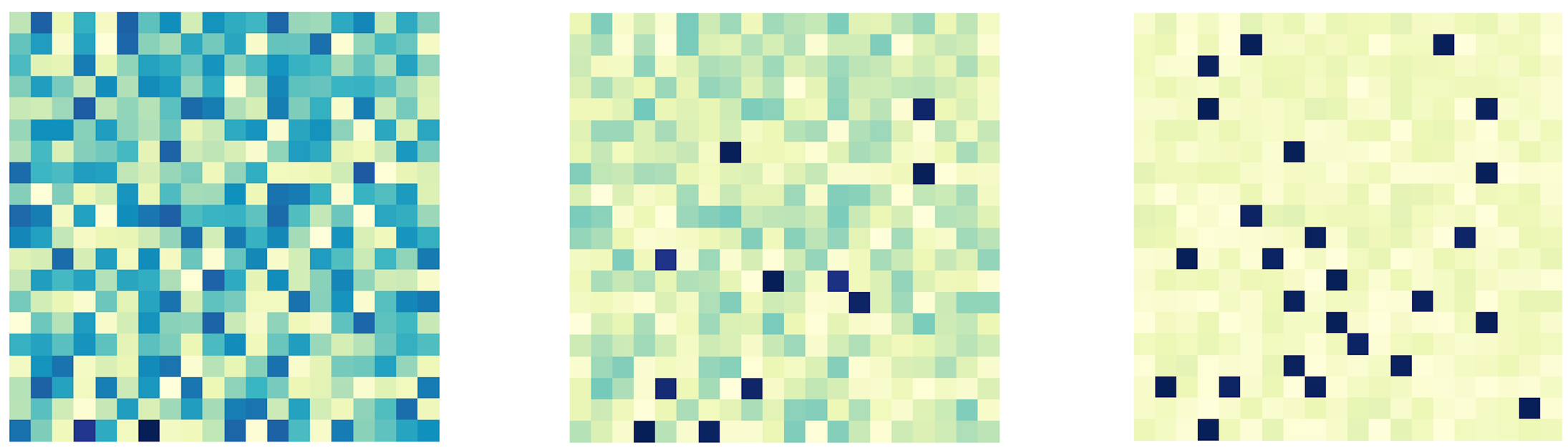}
  \caption{Outputs of the experiment described in Section~\ref{experiments:interpolation}. \textsc{Left Plot}: $\lambda = 1$ (just the submodular objective). \textsc{Right Plot}: $\lambda = 0$ (just the concave objective). \textsc{Middle Plot}: $0<\lambda<1$ (combination of both objectives). A darker square corresponds to a larger entry in the final output of Algorithm~\ref{alg:frank_wolfe} for the corresponding element.}
  \label{fig:lambdafig}
  \vspace{-0.3cm}
\end{figure}

In Figure \ref{fig:lambdafig}, we can see how changing $\lambda$ changes the solution. The plots in the figure show the final output of Algorithm \ref{alg:frank_wolfe} when run on just the submodular objective $G$ (left plot), just the concave objective $C$ (right plot), and a combination of both (middle plot). The algorithm is run with the same cardinality constraint of 25 in all plots, which corresponds to imposing that the $\ell_1$ norm of each iteratation must be at most 25. It is important to note that we represent the exact (continuous) output of the algorithm here. To get a discrete solution, a rounding method should be applied. Also, all of the runs of the algorithm begin from the same fixed starting point inside the cardinality constrained polytope. The step sizes used by the different runs are all constant, and were chosen empirically.

\subsection{Other Submodular + Concave Objectives}

In this section we compare our algorithms with two baselines: the Frank-Wolfe algorithm \citep{frankwolfe1956} and the projected gradient ascent algorithm. In all the experiments done in this section, the objective function is again $F = \lambda G + (1-\lambda) C$, where $G$ and $C$ are a DR-submodular function and a concave function, respectively. For simplicity, we set $\lambda = 1/2$ throughout the section.

\subsubsection{Quadratic Programming} \label{section:quadratic_programming}

In this section, we define $G(\vx) = \frac{1}{2} ~\vx^\top \mathbf{H} \vx + \mathbf{h}^\top \vx + c$. Note that by choosing an appropriate matrix $\mathbf{H}$ and vector $\mathbf{h}$, this objective can be made to be monotone or non-monotone DR-submodular. We also define the down-closed constraint set to be $P = \{ \vx \in \bR_+^n ~|~ \mathbf{Ax \leq b, x  \leq u, A} \in \bR_{+}^{m \times n}, \mathbf{b} \in \bR_+^m \}$. 
Following \cite{bian2017continuous}, we choose the matrix $\mathbf{H} \in \bR^{n\times n}$ to be a randomly generated symmetric matrix with entries uniformly distributed in $[-1,0]$, and the matrix $\mathbf{A}$ to be a random matrix with entries uniformly distributed in $[0.01, 1.01]$ (the $0.01$ addition here is used to ensure that the entries are all positive). The vector $\vb$ is chosen as the all ones vector, and the vector $\vu$ is a tight upper bound on $P$ whose $i^\text{th}$ coordinate is defined as $u_i = \min_{j \in [m]} b_j/\mathbf{A}_{ji}$. We let $\vh = -0.2\cdot \mathbf{H}^\top \vu$ which makes $G$ non-monotone. Finally, although this does not affect the results of our experiments, we take $c$ to be a large enough additive constant (in this case 10) to make $G$ non-negative.

It is know that when the Hessian of a quadratic program is negative semidefinite, the resulting objective is concave. Accordingly, we let $C(\vx) = \frac{1}{20} ~\vx^\top \mathbf{D} \vx$, where $\mathbf{D}$ is a negative semidefinite matrix defined by the negation of the product of a random matrix with entries in $[0,1]$ with its transpose. As one can observe, the generality of DR-submodular + concave objectives allows us to consider quadratic programming with very different Hessians. We hope that our ability to do this will motivate future work about quadratic programming for a broader class of Hessian matrices.

In the current experiment, we let $n \in \{8, 12, 16\}$ and $m \in\{0.5n, n, 1.5n\}$, and run each algorithm for 50 iterations. Note that having fixed the number of iterations,  the maximum step size for Algorithms~\ref{alg:continuous_greedy} and \ref{alg:measured_continuous_greedy} is upper bounded by $($number of iterations$)^{-1} = \nicefrac{1}{50}$ to guarantee that these algorithms remain within the polytope. To ensure consistency, we set the step sizes for the other algorithms to be $\nicefrac{1}{50}$ as well, except for Algorithm~\ref{alg:non_ob_fw} for which we set to the value of $\eps$ given by solving $e^{-1} \cdot \beta(\eps) / \eps^2 = 50$. This ensures that the gradient computation in Algorithm \ref{alg:non_ob_fw} is not too time consuming. We start Algorithms \ref{alg:continuous_greedy} and \ref{alg:measured_continuous_greedy} from the starting point their pseudocodes specify, and the other algorithms from the same arbitrary point. We show the results for the aforementioned values of $n$ and $m$ in Figure \ref{fig:fig}. The appropriate values of $n$ and $m$ are mentioned below each plot, and each plot graphs the average of $50$ runs of the experiment. We also note that since Algorithms \ref{alg:frank_wolfe} and \ref{alg:non_ob_fw} output the best among the results of all their iterations, we just plot the final output of these algorithms instead of the entire trajectory.

\subsubsection{D-optimal Experimental Design} \label{ssc:d_optimal}

Following \citet{chen2018online}, the DR-submodular objective function for the D-optimal experimental design problem is $ G(\vx) = \log \det \left( \sum_{i=1}^n x_i \mathbf{Y_i^\top Y_i} \right)$.  Here, $\mathbf{Y_i}$ is an $n$ dimensional row-vector in which each entry is drawn independently from the standard Gaussian distribution.  The choice of concave function is $C(\vx) = \tfrac{1}{10} \sum_{i=1}^n \log (x_i)$.
In this experiment there is no combinatorial constraint. Instead, we are interested in maximization over a box constraint, i.e., over $[1,2]^n$ (note that the box is shifted compared to the standard $[0, 1]^n$ to ensure that $G$ is defined everywhere as it is undefined at $\vx = \mathbf{0}$).
The final outputs of all the algorithms for $n=8, 12, 16$ appear in Figure~\ref{fig:expt_design_fig}. Like in Section~\ref{section:quadratic_programming}, each algorithm was run for $50$ iterations, and each plot is the average of $50$ runs. The step sizes and starting points used by the algorithms are set exactly like in Section~\ref{section:quadratic_programming}.

\paragraph{Takeaways.} Based on our experiments, we can observe that Algorithms~\ref{alg:continuous_greedy} and~\ref{alg:non_ob_fw} consistently outperform the other algorithms. We can also see (especially in the D-optimal experimental design problem where they almost superimpose) that the difference between Algorithm~\ref{alg:frank_wolfe} and the standard Frank-Wolfe algorithm are minimal, but we believe that the difference between the two algorithms can be made more pronounced by considering settings in which the gradient of $C$ dominates the gradient of $G$. Finally, one can note that the output values in the plots corresponding to the quadratic programming problem  discussed in Section \ref{section:quadratic_programming} tend to decrease when the number of constraints increases, which matches our intuitive expectation.

\section{Conclusion}\label{section:conclusion}

In this paper, we have considered the maximization of a class of objective functions that is strictly larger than both DR-submodular functions and concave functions. The ability to optimize this class of functions using first-order information is interesting from both theoretical  and practical points of view. Our results provide a step towards the goal of efficiently analyzing structured non-convex functions---a goal that is becoming increasingly relevant.


\begin{figure}[H]
\begin{subfigure}{.49\textwidth}
	 \centering
  \includegraphics[width=\textwidth]{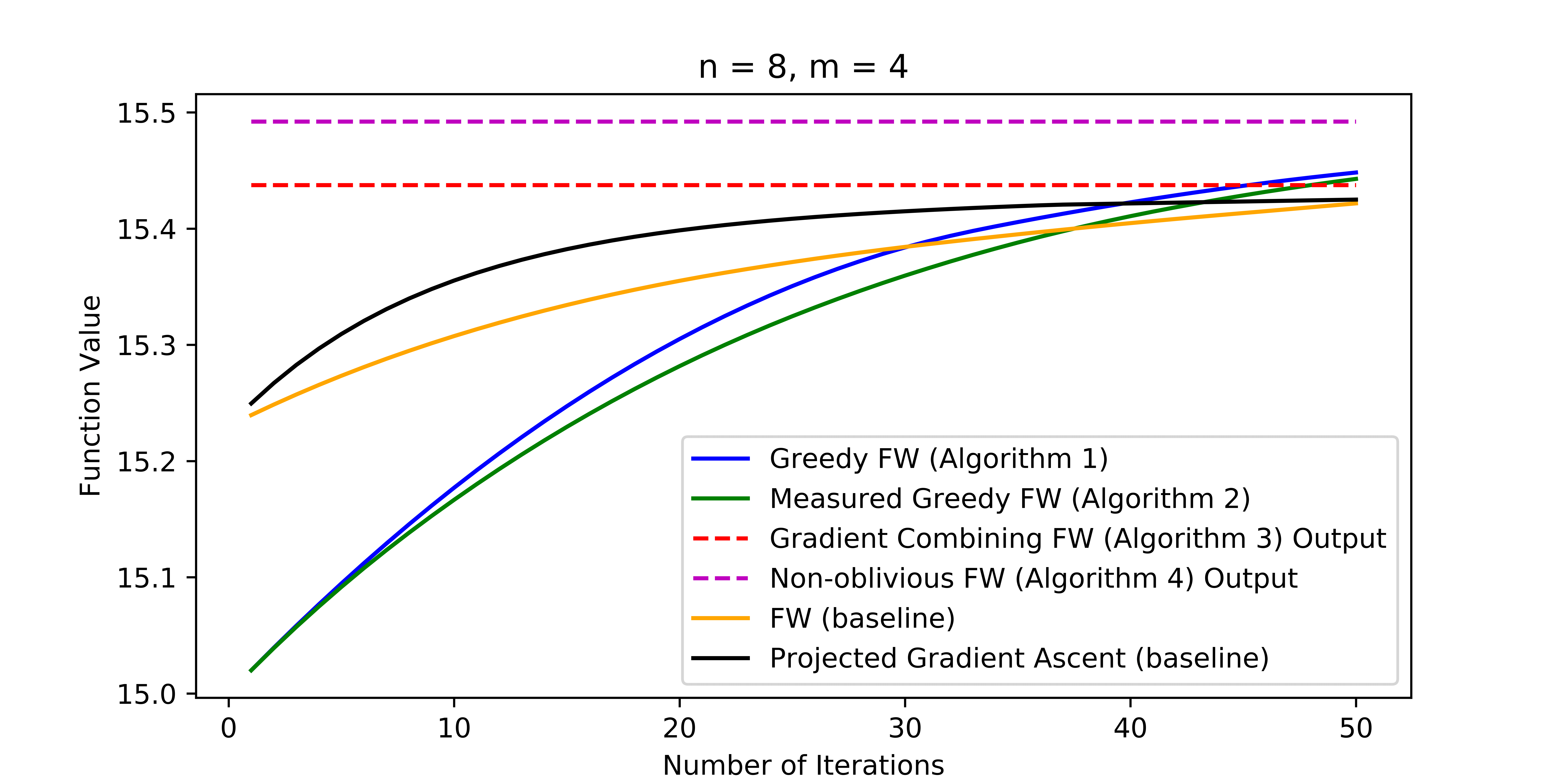}
  \caption{$n=8, m=4$}
  \label{fig:qp_n8m4}
\end{subfigure}
\hfill
\begin{subfigure}{.49\textwidth}
  \centering
  \includegraphics[width=\linewidth]{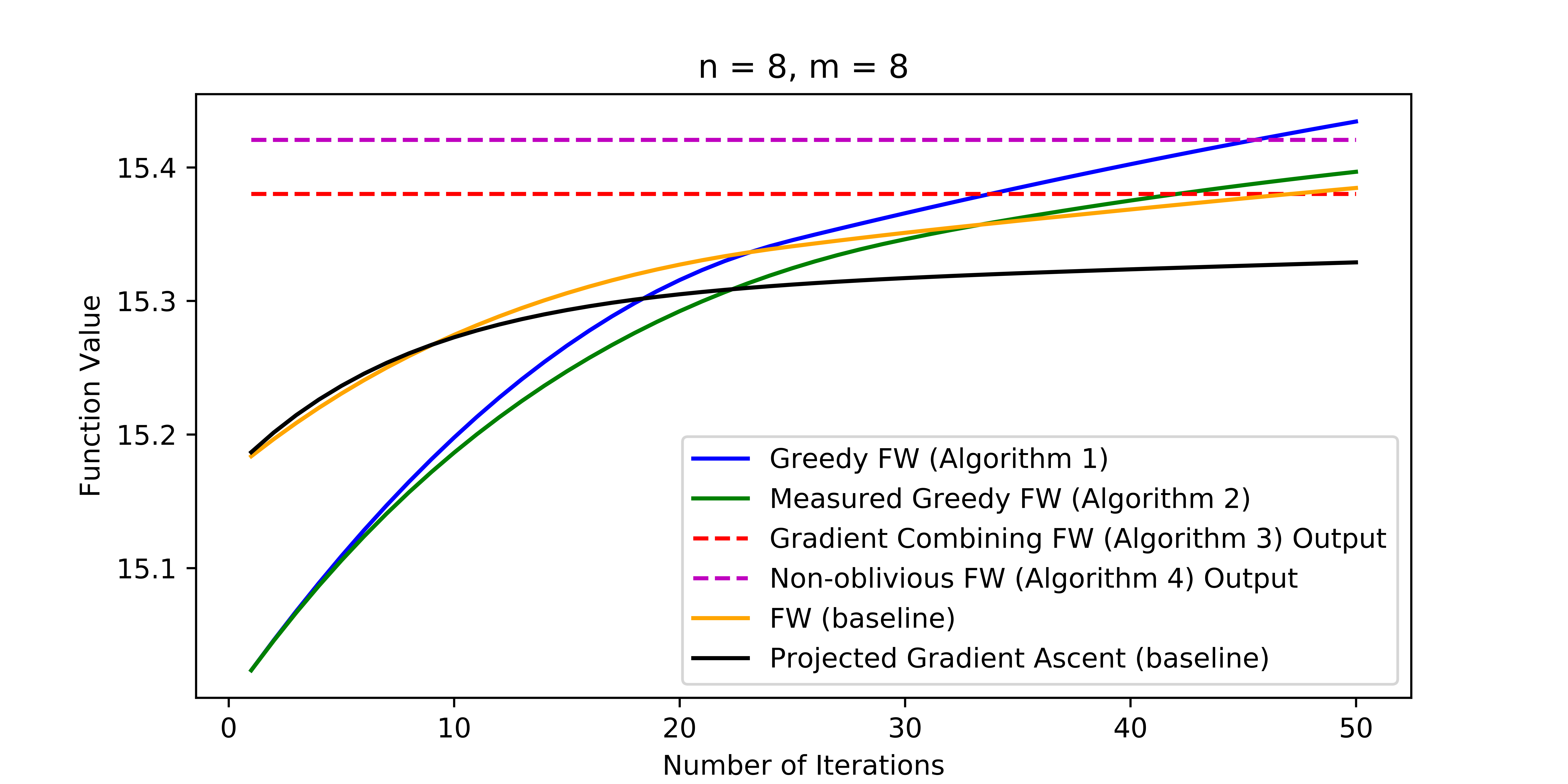}
  \caption{$n=8, m=8$}
  \label{fig:qp_n8m8}
\end{subfigure}

\end{figure}
\begin{figure}\ContinuedFloat

\begin{subfigure}{.49\textwidth}
  \centering
  \includegraphics[width=\linewidth]{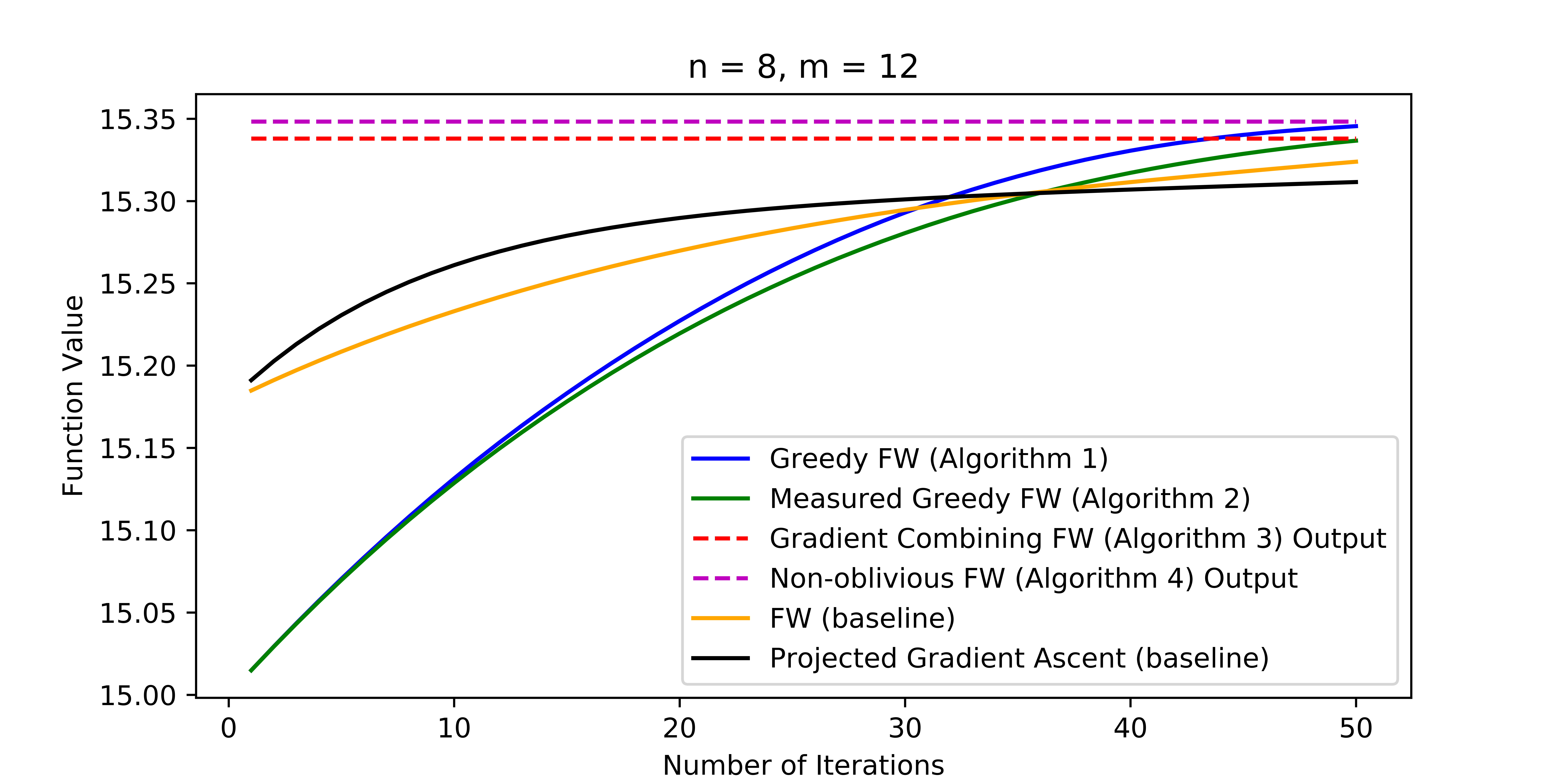}
  \caption{$n=8, m=12$}
  \label{fig:qp_n8m12}
\end{subfigure}
\hfill
\begin{subfigure}{.49\textwidth}
  \centering
  \includegraphics[width=\linewidth]{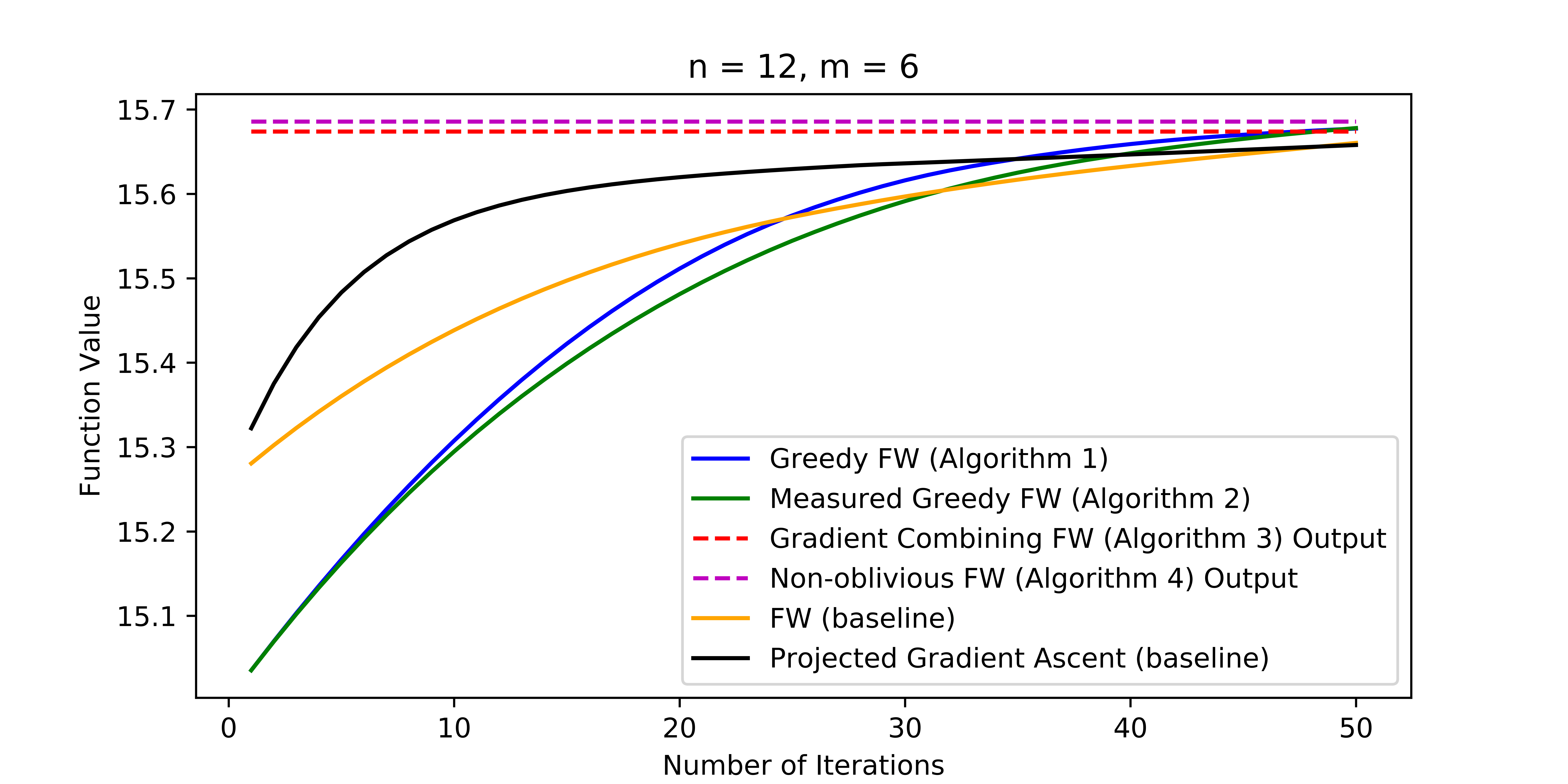}
  \caption{$n=12, m=6$}
\end{subfigure}

\begin{subfigure}{.49\textwidth}
  \centering
  \includegraphics[width=\linewidth]{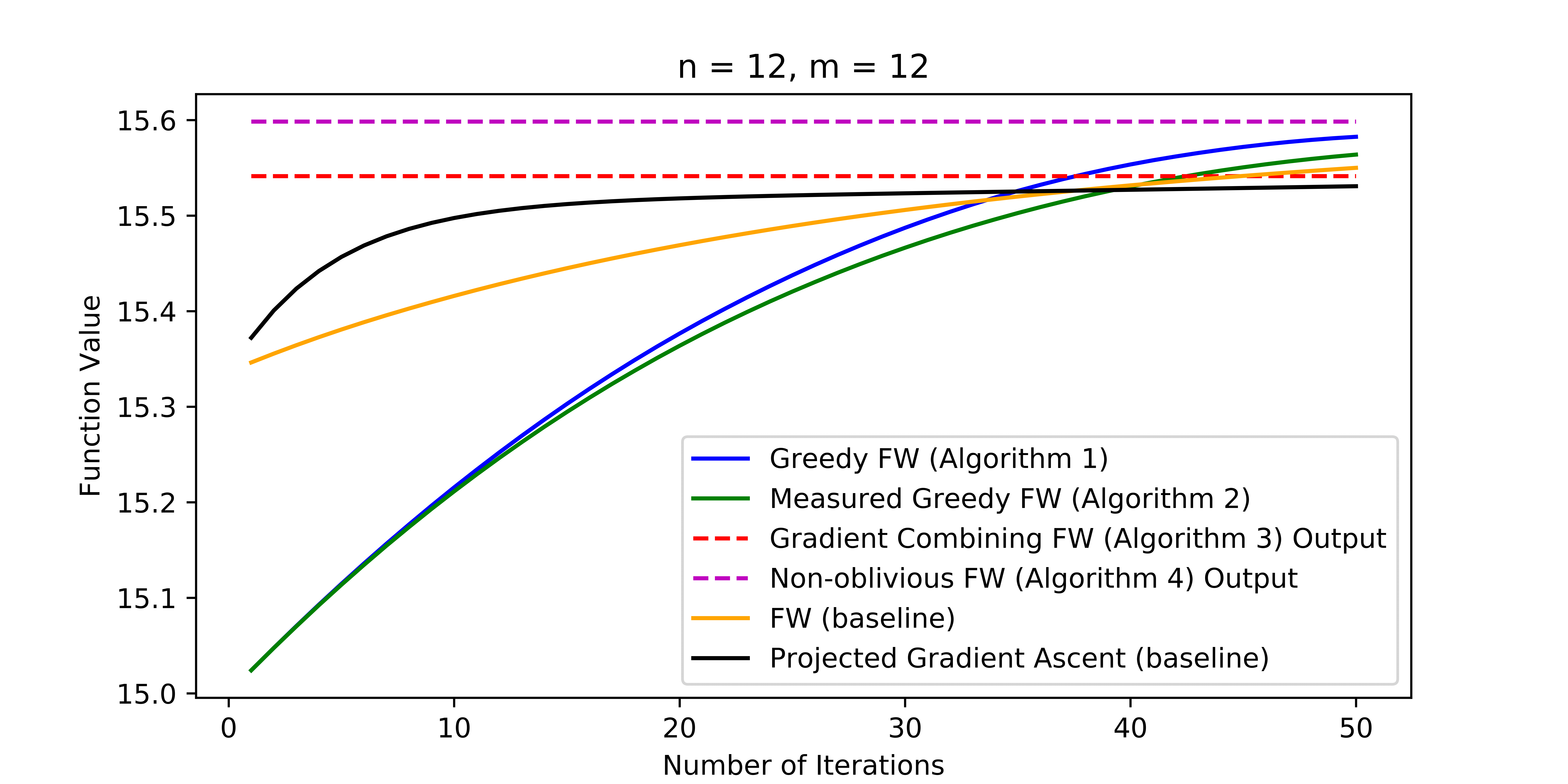}
  \caption{$n=12, m=12$}
\end{subfigure}
\hfill
\begin{subfigure}{.49\textwidth}
  \centering
  \includegraphics[width=\linewidth]{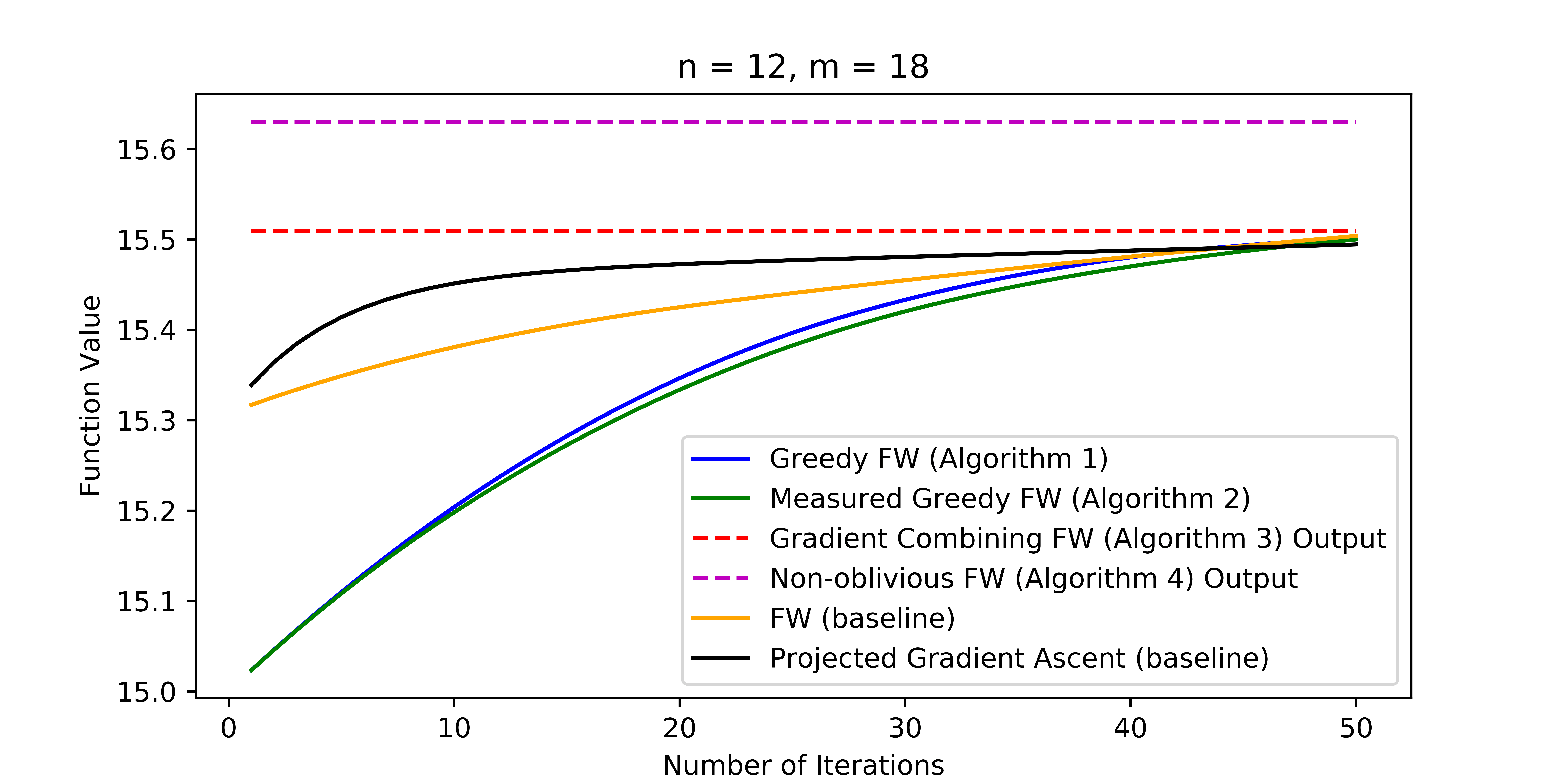}
  \caption{$n=12, m=18$}
\end{subfigure}

\begin{subfigure}{.49\textwidth}
  \centering
  \includegraphics[width=\linewidth]{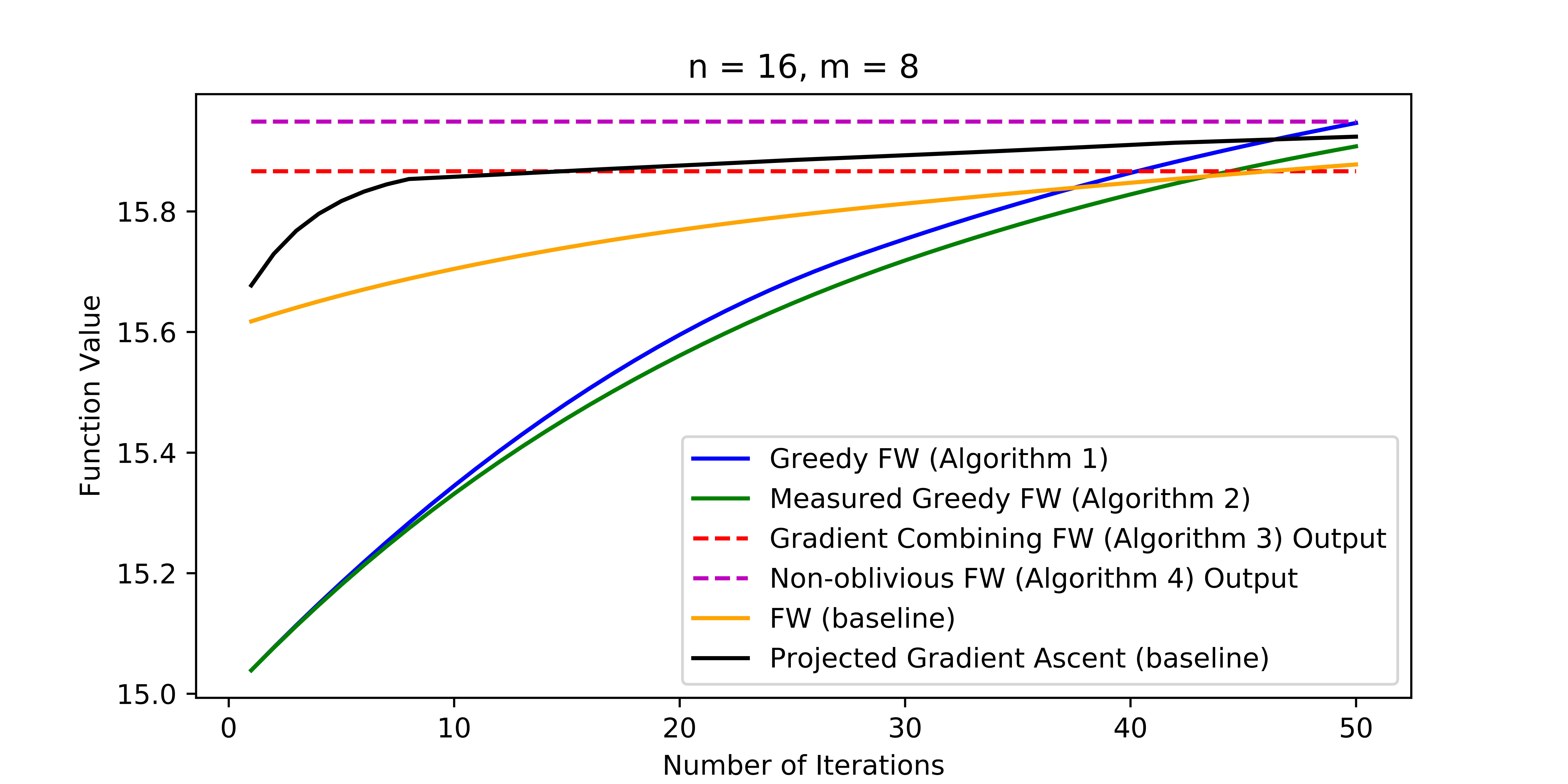}
  \caption{$n=16, m=8$}
\end{subfigure}
\hfill
\begin{subfigure}{.49\textwidth}
  \centering
  \includegraphics[width=\linewidth]{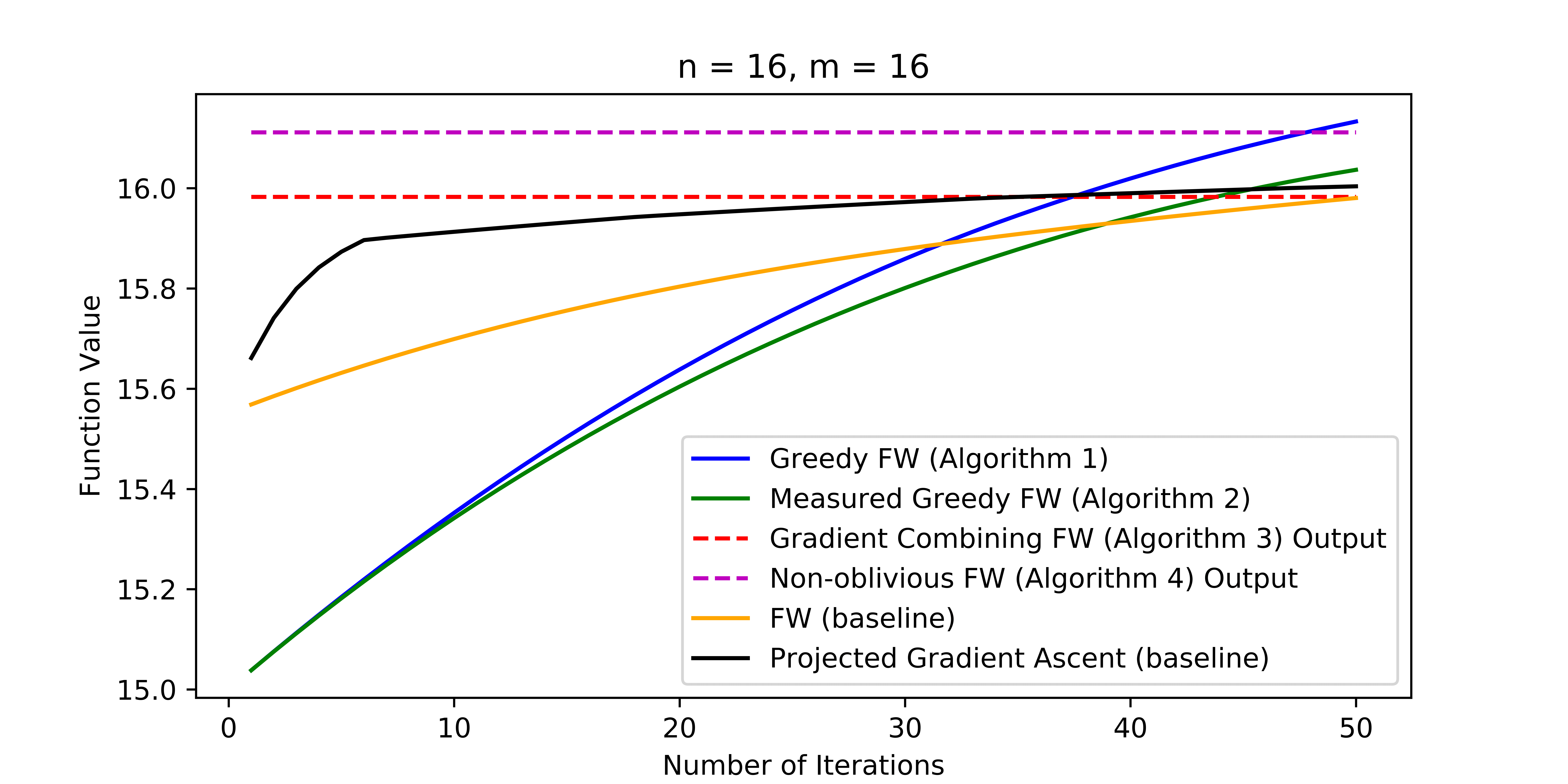}
  \caption{$n=16, m=16$}
\end{subfigure}

\begin{subfigure}{.49\textwidth}
  \centering
  \includegraphics[width=\linewidth]{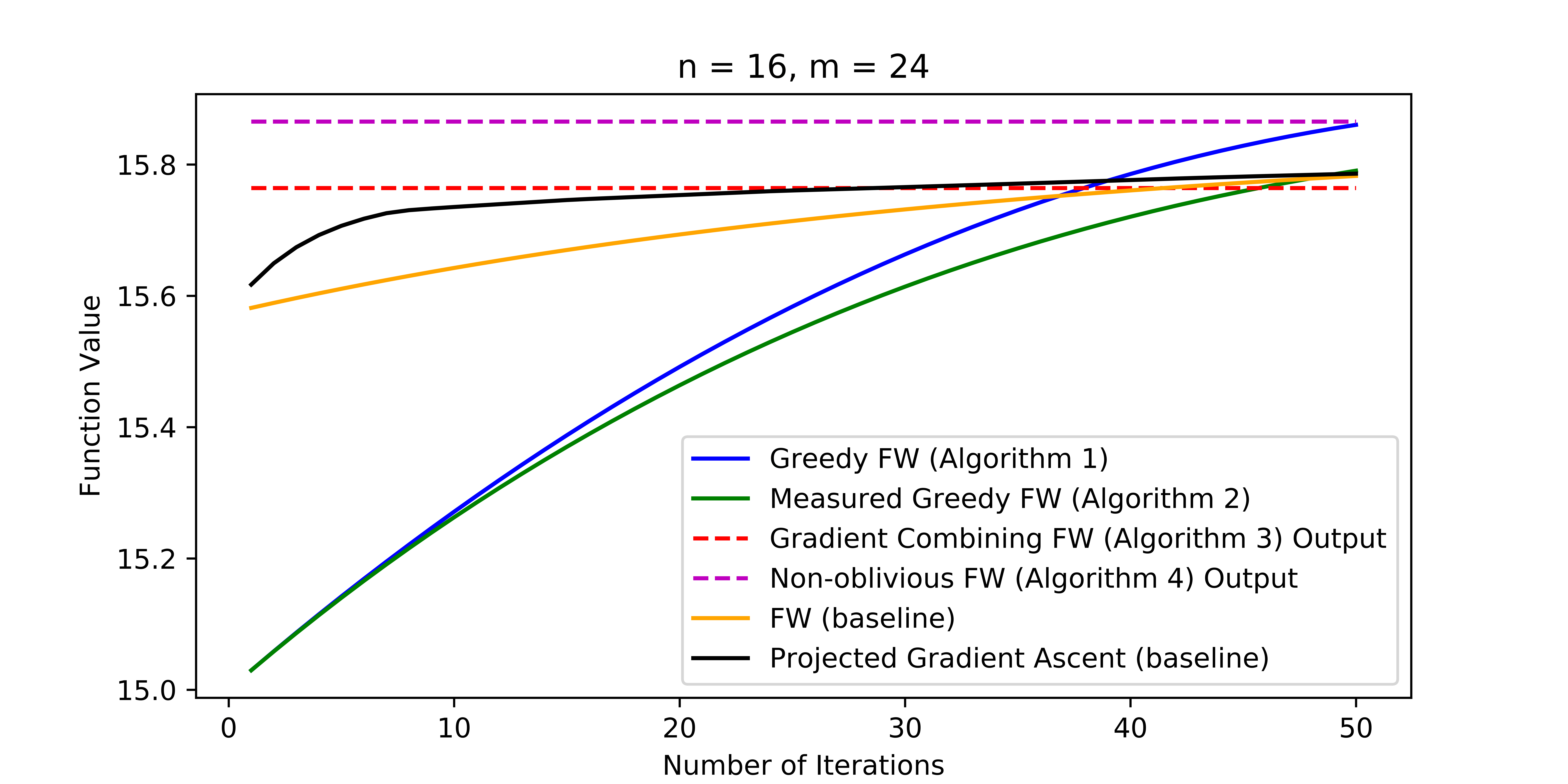}
  \caption{$n=16, m=24$}
\end{subfigure}
\hfill
\begin{subfigure}{.49\textwidth}
  \centering
  \includegraphics[width=\linewidth]{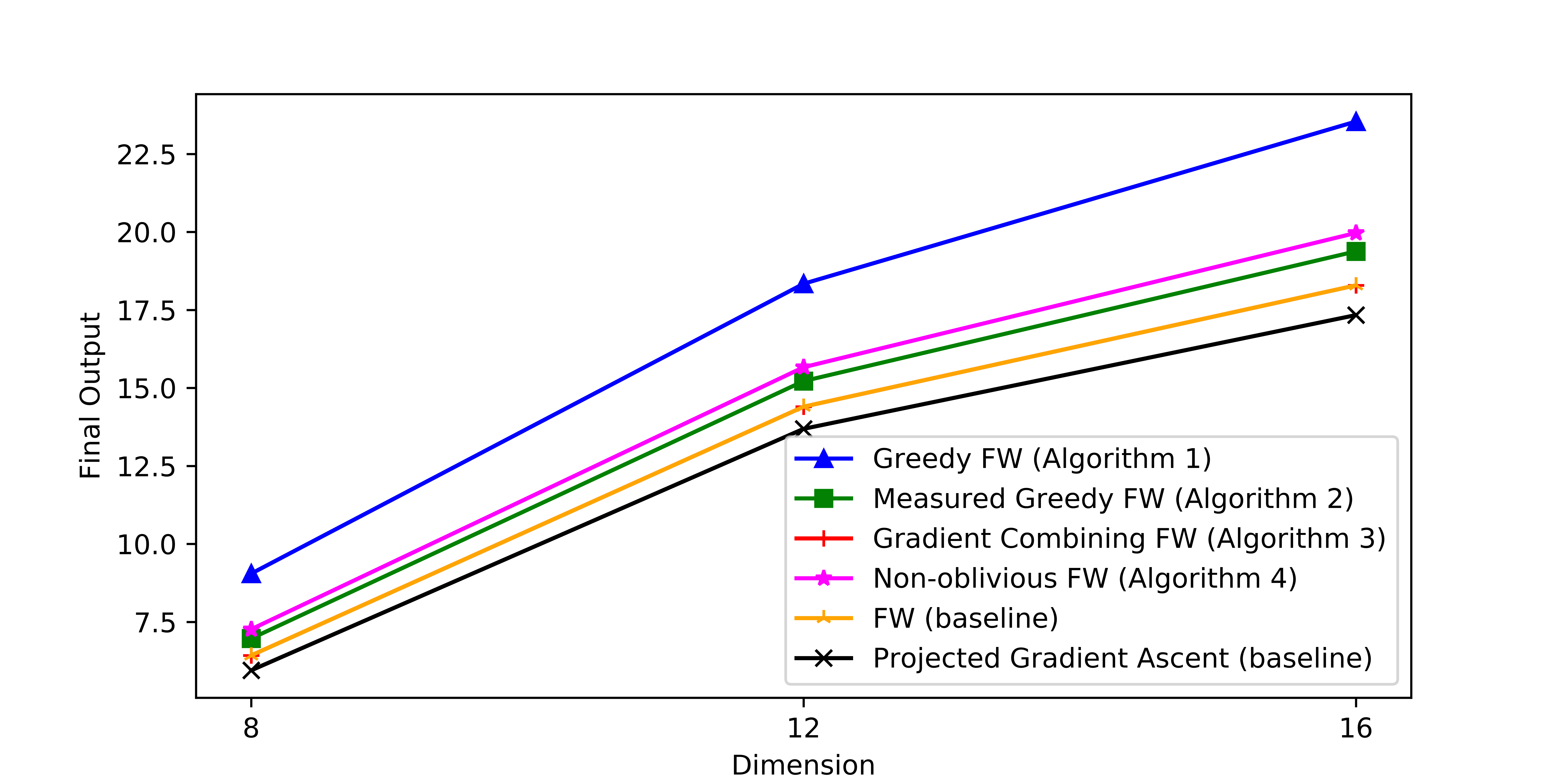}
  \caption{Final outputs for the D-optimal design problem.}
  \label{fig:expt_design_fig}
\end{subfigure}
\caption{All the plots except for \ref{fig:expt_design_fig} correspond to the quadratic programming experiment from Section~\ref{section:quadratic_programming}. For these plots, the dimension ($n$) and the number of constraints ($m$) is mentioned below each plot. Plot~\ref{fig:expt_design_fig} pertains to the D-optimal experimental design problem from Section~\ref{ssc:d_optimal}.}
\label{fig:fig}
\vspace{-0.3cm}
\end{figure}


\bibliographystyle{plainnat}
\bibliography{submodularPlusConcave}

\begin{thebibliography}{63}
\providecommand{\natexlab}[1]{#1}
\providecommand{\url}[1]{\texttt{#1}}
\expandafter\ifx\csname urlstyle\endcsname\relax
  \providecommand{\doi}[1]{doi: #1}\else
  \providecommand{\doi}{doi: \begingroup \urlstyle{rm}\Url}\fi

\bibitem[Alieva et~al.(2021)Alieva, Aceves, Song, Mayo, Yue, and
  Chen]{alieva2021learning}
Ayya Alieva, Aiden Aceves, Jialin Song, Stephen Mayo, Yisong Yue, and Yuxin
  Chen.
\newblock Learning to make decisions via submodular regularization.
\newblock In \emph{International Conference on Learning Representations}, 2021.

\bibitem[Allen-Zhu et~al.(2019)Allen-Zhu, Li, and Liang]{allenzhu2020learning}
Zeyuan Allen-Zhu, Yuanzhi Li, and Yingyu Liang.
\newblock Learning and generalization in overparameterized neural networks,
  going beyond two layers.
\newblock In \emph{Advances in Neural Information Processing Systems}, 2019.

\bibitem[Anari et~al.(2016)Anari, Gharan, and Rezaei]{anari2016monte}
Nima Anari, Shayan~Oveis Gharan, and Alireza Rezaei.
\newblock Monte carlo markov chain algorithms for sampling strongly rayleigh
  distributions and determinantal point processes.
\newblock In \emph{29th Annual Conference on Learning Theory}, 2016.

\bibitem[Arora et~al.(2012)Arora, Ge, Kannan, and Moitra]{arora2011computing}
Sanjeev Arora, Rong Ge, Ravi Kannan, and Ankur Moitra.
\newblock Computing a nonnegative matrix factorization -- provably.
\newblock In \emph{Proceedings of the Forty-Fourth Annual ACM Symposium on
  Theory of Computing}, 2012.

\bibitem[Bian et~al.(2017{\natexlab{a}})Bian, Levy, Krause, and
  Buhmann]{bian2017continuous}
An~Bian, Kfir~Y. Levy, Andreas Krause, and Joachim~M. Buhmann.
\newblock Continuous dr-submodular maximization: Structure and algorithms.
\newblock In \emph{Proceedings of the 31st International Conference on Neural
  Information Processing Systems}, 2017{\natexlab{a}}.

\bibitem[Bian et~al.(2019)Bian, Buhmann, and Krause]{bian2018optimal}
An~Bian, Joachim~M. Buhmann, and Andreas Krause.
\newblock Optimal dr-submodular maximization and applications to provable mean
  field inference.
\newblock In \emph{Proceedings of the 36th International Conference on Machine
  Learning}, 2019.

\bibitem[Bian et~al.(2017{\natexlab{b}})Bian, Mirzasoleiman, Buhmann, and
  Krause]{guaranteed2017bian}
Andrew~An Bian, Baharan Mirzasoleiman, Joachim~M. Buhmann, and Andreas Krause.
\newblock Guaranteed non-convex optimization: Submodular maximization over
  continuous domains.
\newblock In \emph{AISTATS}, 2017{\natexlab{b}}.

\bibitem[Bresler et~al.(2019)Bresler, Koehler, Moitra, and
  Mossel]{bresler2018learning}
Guy Bresler, Frederic Koehler, Ankur Moitra, and Elchanan Mossel.
\newblock Learning restricted boltzmann machines via influence maximization.
\newblock In \emph{Proceedings of the 51st Annual ACM SIGACT Symposium on
  Theory of Computing}, 2019.

\bibitem[Bubeck(2015)]{bubeck2015convex}
S\'{e}bastien Bubeck.
\newblock Convex optimization: Algorithms and complexity.
\newblock \emph{Found. Trends Mach. Learn.}, 2015.

\bibitem[C{\u{a}}linescu et~al.(2011)C{\u{a}}linescu, Chekuri, P{\'{a}}l, and
  Vondr{\'{a}}k]{maximizing2011calinescu}
Gruia C{\u{a}}linescu, Chandra Chekuri, Martin P{\'{a}}l, and Jan
  Vondr{\'{a}}k.
\newblock Maximizing a monotone submodular function subject to a matroid
  constraint.
\newblock \emph{{SIAM} J. Comput.}, 2011.

\bibitem[Chen et~al.(2018)Chen, Hassani, and Karbasi]{chen2018online}
Lin Chen, Hamed Hassani, and Amin Karbasi.
\newblock Online continuous submodular maximization.
\newblock In \emph{Proceedings of the Twenty-First International Conference on
  Artificial Intelligence and Statistics}, 2018.

\bibitem[Chen et~al.(2019{\natexlab{a}})Chen, Feldman, and
  Karbasi]{chen2018unconstrained}
Lin Chen, Moran Feldman, and Amin Karbasi.
\newblock Unconstrained submodular maximization with constant adaptive
  complexity.
\newblock In \emph{Proceedings of the 51st Annual {ACM} {SIGACT} Symposium on
  Theory of Computing}, STOC 2019, 2019{\natexlab{a}}.

\bibitem[Chen et~al.(2019{\natexlab{b}})Chen, Chi, Fan, and Ma]{Chen_2019}
Yuxin Chen, Yuejie Chi, Jianqing Fan, and Cong Ma.
\newblock Gradient descent with random initialization: fast global convergence
  for nonconvex phase retrieval.
\newblock \emph{Mathematical Programming}, 2019{\natexlab{b}}.

\bibitem[Derezinski et~al.(2020)Derezinski, Liang, and
  Mahoney]{derezinski20experimental}
Michal Derezinski, Feynman Liang, and Michael Mahoney.
\newblock Bayesian experimental design using regularized determinantal point
  processes.
\newblock In \emph{Proceedings of the Twenty Third International Conference on
  Artificial Intelligence and Statistics}, 2020.

\bibitem[Djolonga and Krause(2014)]{josip2014map}
Josip Djolonga and Andreas Krause.
\newblock From map to marginals: Variational inference in bayesian submodular
  models.
\newblock In \emph{Advances in Neural Information Processing Systems}, 2014.

\bibitem[Du et~al.(2019)Du, Zhai, Poczos, and Singh]{du2019gradient}
Simon~S. Du, Xiyu Zhai, Barnabas Poczos, and Aarti Singh.
\newblock Gradient descent provably optimizes over-parameterized neural
  networks.
\newblock In \emph{International Conference on Learning Representations}, 2019.

\bibitem[Dwivedi et~al.(2019)Dwivedi, Chen, Wainwright, and
  Yu]{dwivedi2019logconcave}
Raaz Dwivedi, Yuansi Chen, Martin~J. Wainwright, and Bin Yu.
\newblock Log-concave sampling: Metropolis-hastings algorithms are fast.
\newblock \emph{Journal of Machine Learning Research}, 2019.

\bibitem[Elenberg et~al.(2016)Elenberg, Khanna, Dimakis, and
  Negahban]{elenberg2017restricted}
Ethan~R. Elenberg, Rajiv Khanna, Alexandros~G. Dimakis, and Sahand Negahban.
\newblock Restricted strong convexity implies weak submodularity.
\newblock In \emph{Proceedings of the 30th Conference on Neural Information
  Processing Systems}, 2016.

\bibitem[Elenberg et~al.(2017)Elenberg, Dimakis, Feldman, and
  Karbasi]{elenberg2017streaming}
Ethan~R. Elenberg, Alexandros~G. Dimakis, Moran Feldman, and Amin Karbasi.
\newblock Streaming weak submodularity: Interpreting neural networks on the
  fly.
\newblock In \emph{Proceedings of the 31st International Conference on Neural
  Information Processing Systems}, 2017.

\bibitem[Ene and Nguyen(2019)]{ene2019parallel}
Alina Ene and Huy~L. Nguyen.
\newblock Parallel algorithm for non-monotone dr-submodular maximization.
\newblock In \emph{Proceedings of the 37th International Conference on Machine
  Learning}, 2019.

\bibitem[Ene et~al.(2020)Ene, Nikolakaki, and Terzi]{team2020ene}
Alina Ene, Sofia~Maria Nikolakaki, and Evimaria Terzi.
\newblock Team formation: Striking a balance between coverage and cost.
\newblock \emph{CoRR}, 2020.

\bibitem[Feldman(2021)]{guess2021feldman}
Moran Feldman.
\newblock Guess free maximization of submodular and linear sums.
\newblock \emph{Algorithmica}, 2021.

\bibitem[Feldman et~al.(2011)Feldman, Naor, and Schwartz]{feldman2011greedy}
Moran Feldman, Joseph Naor, and Roy Schwartz.
\newblock A unified continuous greedy algorithm for submodular maximization.
\newblock In \emph{2011 IEEE 52nd Annual Symposium on Foundations of Computer
  Science}, 2011.

\bibitem[Filmus and Ward(2012)]{filmus2012matroid}
Yuval Filmus and Justin Ward.
\newblock A tight combinatorial algorithm for submodular maximization subject
  to a matroid constraint.
\newblock In \emph{2012 IEEE 53rd Annual Symposium on Foundations of Computer
  Science}, 2012.

\bibitem[Frank and Wolfe(1956)]{frankwolfe1956}
Marguerite Frank and Philip Wolfe.
\newblock An algorithm for quadratic programming.
\newblock \emph{Naval Research Logistics Quarterly}, 1956.

\bibitem[Ge et~al.(2015)Ge, Huang, Jin, and Yuan]{ge2015escaping}
Rong Ge, Furong Huang, Chi Jin, and Yang Yuan.
\newblock Escaping from saddle points --- online stochastic gradient for tensor
  decomposition.
\newblock In \emph{Proceedings of The 28th Conference on Learning Theory},
  2015.

\bibitem[Gillenwater et~al.(2012)Gillenwater, Kulesza, and
  Taskar]{gillenwater2012near}
J.~Gillenwater, A.~Kulesza, and B.~Taskar.
\newblock {Near-Optimal MAP Inference for Determinantal Point Processes}.
\newblock In \emph{Proc.\ Neural Information Processing Systems (NIPS)}, 2012.

\bibitem[Harshaw et~al.(2019)Harshaw, Feldman, Ward, and
  Karbasi]{submodular2019harshaw}
Chris Harshaw, Moran Feldman, Justin Ward, and Amin Karbasi.
\newblock Submodular maximization beyond non-negativity: Guarantees, fast
  algorithms, and applications.
\newblock In \emph{ICML}, 2019.

\bibitem[Hassani et~al.(2017)Hassani, Soltanolkotabi, and
  Karbasi]{hassani2017gradientmethods}
Hamed Hassani, Mahdi Soltanolkotabi, and Amin Karbasi.
\newblock Gradient methods for submodular maximization.
\newblock In \emph{Proceedings of the 31st International Conference on Neural
  Information Processing Systems}, 2017.

\bibitem[Hassani et~al.(2020{\natexlab{a}})Hassani, Karbasi, Mokhtari, and
  Shen]{hassani2020nonconvexand}
Hamed Hassani, Amin Karbasi, Aryan Mokhtari, and Zebang Shen.
\newblock Stochastic conditional gradient++: (non)convex minimization and
  continuous submodular maximization.
\newblock \emph{SIAM Journal on Optimization}, 2020{\natexlab{a}}.

\bibitem[Hassani et~al.(2020{\natexlab{b}})Hassani, Karbasi, Mokhtari, and
  Shen]{hassani2020stochastic}
Hamed Hassani, Amin Karbasi, Aryan Mokhtari, and Zebang Shen.
\newblock Stochastic conditional gradient++, 2020{\natexlab{b}}.

\bibitem[Hazan et~al.(2016)Hazan, Levy, and Shalev-Shwartz]{hazan2015graduated}
Elad Hazan, Kfir~Y. Levy, and Shai Shalev-Shwartz.
\newblock On graduated optimization for stochastic non-convex problems.
\newblock In \emph{Proceedings of The 33rd International Conference on Machine
  Learning}, 2016.

\bibitem[He(2010)]{he2010laplacianregularized}
Xiaofei He.
\newblock Laplacian regularized d-optimal design for active learning and its
  application to image retrieval.
\newblock \emph{IEEE Transactions on Image Processing}, 2010.

\bibitem[Jacot et~al.(2018)Jacot, Gabriel, and Hongler]{jacot2020neural}
Arthur Jacot, Franck Gabriel, and Clément Hongler.
\newblock Neural tangent kernel: Convergence and generalization in neural
  networks.
\newblock In \emph{Proceedings of the 32nd International Conference on Neural
  Information Processing Systems}, 2018.

\bibitem[Kazemi et~al.(2020{\natexlab{a}})Kazemi, Minaee, Feldman, and
  Karbasi]{kazemi2020regularized}
Ehsan Kazemi, Shervin Minaee, Moran Feldman, and Amin Karbasi.
\newblock Regularized submodular maximization at scale.
\newblock \emph{CoRR}, 2020{\natexlab{a}}.

\bibitem[Kazemi et~al.(2020{\natexlab{b}})Kazemi, Minaee, Feldman, and
  Karbasi]{regularized2020kazemi}
Ehsan Kazemi, Shervin Minaee, Moran Feldman, and Amin Karbasi.
\newblock Regularized submodular maximization at scale.
\newblock \emph{CoRR}, 2020{\natexlab{b}}.

\bibitem[Kempe et~al.(2003)Kempe, Kleinberg, and Tardos]{kempe2013spread}
David Kempe, Jon Kleinberg, and \'{E}va Tardos.
\newblock Maximizing the spread of influence through a social network.
\newblock In \emph{Proceedings of the Ninth {ACM} {SIGKDD} International
  Conference on Knowledge Discovery and Data Mining}, KDD '03. Association for
  Computing Machinery, 2003.

\bibitem[Krause and Cevher(2010)]{krause2010dictionary}
Andreas Krause and Volkan Cevher.
\newblock Submodular dictionary selection for sparse representation.
\newblock In \emph{Proceedings of the 27th International Conference on
  International Conference on Machine Learning}, 2010.

\bibitem[Kulesza(2012)]{Kulesza_2012}
Alex Kulesza.
\newblock Determinantal point processes for machine learning.
\newblock \emph{Foundations and Trends® in Machine Learning}, 2012.

\bibitem[Lattimore and Szepesvari(2020)]{banditalgsbook}
Tor Lattimore and Csaba Szepesvari.
\newblock \emph{Bandit algorithms}.
\newblock Cambridge University Press, 2020.

\bibitem[Mirzasoleiman et~al.(2013)Mirzasoleiman, Karbasi, Sarkar, and
  Krause]{mirza2013distributed}
Baharan Mirzasoleiman, Amin Karbasi, Rik Sarkar, and Andreas Krause.
\newblock Distributed submodular maximization: Identifying representative
  elements in massive data.
\newblock In \emph{Advances in Neural Information Processing Systems}, 2013.

\bibitem[Mokhtari et~al.(2018{\natexlab{a}})Mokhtari, Hassani, and
  Karbasi]{mokhtari2017conditional}
Aryan Mokhtari, Hamed Hassani, and Amin Karbasi.
\newblock Conditional gradient method for stochastic submodular maximization:
  Closing the gap.
\newblock In \emph{Proceedings of the Twenty-First International Conference on
  Artificial Intelligence and Statistics}, 2018{\natexlab{a}}.

\bibitem[Mokhtari et~al.(2018{\natexlab{b}})Mokhtari, Hassani, and
  Karbasi]{mokhtari2018decentralized}
Aryan Mokhtari, Hamed Hassani, and Amin Karbasi.
\newblock Decentralized submodular maximization: Bridging discrete and
  continuous settings.
\newblock In \emph{Proceedings of the 35th International Conference on Machine
  Learning}, 2018{\natexlab{b}}.

\bibitem[Murty and Kabadi(1987)]{Murty1987SomeNP}
K.~G. Murty and S.~N. Kabadi.
\newblock Some np-complete problems in quadratic and nonlinear programming.
\newblock \emph{Mathematical Programming}, 1987.

\bibitem[Nesterov(2018)]{nesterovbook}
Yurii Nesterov.
\newblock \emph{Lectures on Convex Optimization}.
\newblock Springer Publishing Company, Incorporated, 2018.

\bibitem[Netrapalli et~al.(2014)Netrapalli, Niranjan, Sanghavi, Anandkumar, and
  Jain]{netrapalli2014nonconvex}
Praneeth Netrapalli, U~N Niranjan, Sujay Sanghavi, Animashree Anandkumar, and
  Prateek Jain.
\newblock Non-convex robust pca.
\newblock In \emph{Advances in Neural Information Processing Systems}, 2014.

\bibitem[Niazadeh et~al.(2018)Niazadeh, Roughgarden, and
  Wang]{niazadeh2018optimal}
Rad Niazadeh, Tim Roughgarden, and Joshua~R. Wang.
\newblock Optimal algorithms for continuous non-monotone submodular and
  dr-submodular maximization.
\newblock In \emph{Proceedings of the 32nd International Conference on Neural
  Information Processing Systems}, 2018.

\bibitem[Raut et~al.(2021)Raut, Sadeghi, and Fazel]{raut2021online}
Prasanna~Sanjay Raut, Omid Sadeghi, and Maryam Fazel.
\newblock Online dr-submodular maximization with stochastic cumulative
  constraints, 2021.

\bibitem[Rebeschini and Karbasi(2015)]{Rebeschini2015fastmixing}
Patrick Rebeschini and Amin Karbasi.
\newblock Fast mixing for discrete point processes.
\newblock In \emph{Proceedings of The 28th Conference on Learning Theory},
  2015.

\bibitem[Robinson et~al.(2019)Robinson, Sra, and Jegelka]{robinson2019flexible}
Joshua Robinson, Suvrit Sra, and Stefanie Jegelka.
\newblock Flexible modeling of diversity with strongly log-concave
  distributions.
\newblock In \emph{Advances in Neural Information Processing Systems}, 2019.

\bibitem[Sadeghi and Fazel(2019)]{sadeghi2019online}
Omid Sadeghi and Maryam Fazel.
\newblock Online continuous {DR}-submodular maximization with long-term budget
  constraints.
\newblock In \emph{Proceedings of the Twenty Third International Conference on
  Artificial Intelligence and Statistics}, 2019.

\bibitem[Sadeghi et~al.(2020)Sadeghi, Raut, and Fazel]{sadeghi2020online}
Omid Sadeghi, Prasanna Raut, and Maryam Fazel.
\newblock A single recipe for online submodular maximization with adversarial
  or stochastic constraints.
\newblock In \emph{Advances in Neural Information Processing Systems}, 2020.

\bibitem[Singla et~al.(2014)Singla, Bogunovic, Bartók, Karbasi, and
  Krause]{singla2014nearoptimally}
Adish Singla, Ilija Bogunovic, Gábor Bartók, Amin Karbasi, and Andreas
  Krause.
\newblock Near-optimally teaching the crowd to classify.
\newblock In \emph{Proceedings of the 31st International Conference on
  International Conference on Machine Learning}, 2014.

\bibitem[Staib and Jegelka(2017)]{staib2017robust}
Matthew Staib and Stefanie Jegelka.
\newblock Robust budget allocation via continuous submodular functions.
\newblock In \emph{Proceedings of the 34th International Conference on Machine
  Learning}, 2017.

\bibitem[Sun et~al.(2016)Sun, Qu, and Wright]{sun2016nonconvex}
Ju~Sun, Qing Qu, and John Wright.
\newblock When are nonconvex problems not scary?, 2016.

\bibitem[Sun and Luo(2016)]{Sun_2016}
Ruoyu Sun and Zhi-Quan Luo.
\newblock Guaranteed matrix completion via non-convex factorization.
\newblock \emph{IEEE Transactions on Information Theory}, 2016.

\bibitem[Sviridenko et~al.(2017)Sviridenko, Vondr{\'{a}}k, and
  Ward]{optimal2017sviridenko}
Maxim Sviridenko, Jan Vondr{\'{a}}k, and Justin Ward.
\newblock Optimal approximation for submodular and supermodular optimization
  with bounded curvature.
\newblock \emph{Math. Oper. Res.}, 2017.

\bibitem[Tohidi et~al.(2020)Tohidi, Amiri, Coutino, Gesbert, Leus, and
  Karbasi]{submodularsurvey}
Ehsan Tohidi, Rouhollah Amiri, Mario Coutino, David Gesbert, Geert Leus, and
  Amin Karbasi.
\newblock Submodularity in action: From machine learning to signal processing
  applications.
\newblock \emph{IEEE Signal Processing Magazine}, 2020.

\bibitem[Wei et~al.(2015)Wei, Iyer, and Bilmes]{wei2015active}
Kai Wei, Rishabh Iyer, and Jeff Bilmes.
\newblock Submodularity in data subset selection and active learning.
\newblock In \emph{Proceedings of the 32nd International Conference on
  International Conference on Machine Learning - Volume 37}, 2015.

\bibitem[Wilder(2018{\natexlab{a}})]{wilder2018}
Bryan Wilder.
\newblock Risk-sensitive submodular optimization.
\newblock \emph{Proceedings of the AAAI Conference on Artificial Intelligence},
  2018{\natexlab{a}}.

\bibitem[Wilder(2018{\natexlab{b}})]{wilder2018game}
Bryan Wilder.
\newblock Equilibrium computation and robust optimization in zero sum games
  with submodular structure.
\newblock \emph{Proceedings of the AAAI Conference on Artificial Intelligence},
  2018{\natexlab{b}}.

\bibitem[Xie et~al.(2019)Xie, Zhang, Shen, Mi, and Qian]{xie2019decentralized}
Jiahao Xie, Chao Zhang, Zebang Shen, Chao Mi, and Hui Qian.
\newblock Decentralized gradient tracking for continuous dr-submodular
  maximization.
\newblock In \emph{Proceedings of the Twenty-Second International Conference on
  Artificial Intelligence and Statistics}, 2019.

\bibitem[Zhang et~al.(2019)Zhang, Chen, Hassani, and Karbasi]{zhang2019online}
Mingrui Zhang, Lin Chen, Hamed Hassani, and Amin Karbasi.
\newblock Online continuous submodular maximization: From full-information to
  bandit feedback.
\newblock In \emph{Advances in Neural Information Processing Systems}, 2019.

\end{thebibliography}

\end{document}